\documentclass[a4paper,10pt]{article}
\pdfoutput=1
\usepackage[english]{babel}
\usepackage{amssymb}
\usepackage{amsmath}
\usepackage{amsthm}
\usepackage{graphicx}
\usepackage{fancyhdr}
\newfont{\bssten}{cmssbx10}
\newfont{\bssnine}{cmssbx10 scaled 900}
\newfont{\bssdoz}{cmssbx10 scaled 1200}

\usepackage{latexsym,amssymb,amsthm,amsxtra}
\usepackage{amsmath}
\usepackage{stmaryrd}
\usepackage{mathrsfs}
\usepackage{tikz} 
\usepackage{pgf} 
{\bf}{\it}
\newtheorem{theorem}{Theorem}
\newtheorem{definition}{Definition}
\newtheorem{lemma}{Lemma}
\newtheorem{remark}{Remark}
\newtheorem{proposition}{Proposition}

\newtheorem{ex}{Example}

\makeatletter
\DeclareRobustCommand{\cev}[1]{%
  \mathpalette\do@cev{#1}%
}
\newcommand{\do@cev}[2]{%
  \fix@cev{#1}{+}%
  \reflectbox{$\m@th#1\vec{\reflectbox{$\fix@cev{#1}{-}\m@th#1#2\fix@cev{#1}{+}$}}$}%
  \fix@cev{#1}{-}%
}
\newcommand{\fix@cev}[2]{%
  \ifx#1\displaystyle
    \mkern#23mu
  \else
    \ifx#1\textstyle
      \mkern#23mu
    \else
      \ifx#1\scriptstyle
        \mkern#22mu
      \else
        \mkern#22mu
      \fi
    \fi
  \fi
}

\makeatother

\def\epsilon{\varepsilon}

\usepackage{anysize}
\marginsize{3cm}{2.5cm}{2.5cm}{2.5cm} 
\newcommand{\be}{\begin{equation}}
\newcommand{\ee}{\end{equation}}

\newcommand\om{{\omega}}

\def\R{{\mathbb R}}
\def\Z{{\mathbb Z}}

\def\N{{\mathbb N}}

\def\maC{{\mathcal{C}}}

\def\T{{\mathcal{T}}}

\def\sl{\scriptsize{\mbox{\textsc{l}}}}
\def\sf{\scriptsize{\mbox{\textsc{f}}}}

\def\maS{{\mathcal S}}

\def\maA{{\mathcal{A}}}
\def\masA{{\mathscr{A}}}

\def\maV{{\mathcal{V}}}
\def\maB{{\mathcal{B}}}

\def\maH{{\mathcal{H}}}
\def\maQ{{\mathscr{Q}}}
\def\maM{{\mathscr{M}}}
\def\maP{{\mathscr{P}}}
\def\maE{{\mathcal{E}}}

\def\maG{{\mathcal{G}}}
\def\maH{{\mathcal{H}}}
\def\mw{{\mathbf w}}
\def\mz{{\mathbf z}}
\def\mc{{\mathbf c}}
\def\ms{{\mathbf s}}
\def\msg{{\mathbf a}}
\def\mga{{\mathbf b}}
\def\msg{{\boldsymbol{\sigma}}}
\def\mga{{\boldsymbol{\gamma}}}

\def\v{{\--}}
\def\pv{{\not\!\!\--}}

\def\T2a{{\tau_{2\vec \alpha^{+}}}}
\def\t2a{{t_{2\vec \alpha^{+}}}}
\newcommand\suite[1]{\left(#1\right)_{n\in\N}}

\newcommand\bp{{\mathbb P^0}}
\newcommand\bpr[1]{{\mathbb P^0}\left[#1\right]}

\def\\mga{{\cal{\mga}}}

\def\({{\Bigl(}}
\def\){{\Bigr)}}

\def\square{\ifmmode\sqr\else{$\sqr$}\fi}
\def\sqr{\vcenter{
         \hrule height.1mm
         \hbox{\vrule width.1mm height2.2mm\kern2.18mm\vrule width.1mm}
         \hrule height.1mm}}                  

\newcommand\ccc{\circledcirc}

\def\Mlcfs{M_{\textsc{lcfs}}}
\def\Qlcfs{Q_{\textsc{lcfs}}}
\def\Clcfs{C_{\textsc{lcfs}}}
\def\Slcfs{S_{\textsc{lcfs}}}
\newenvironment{itemi}
{%
  \begin{list}{$\bullet$}%
  {\noindent%
    \usecounter{enumi}%
    \setlength{\topsep}{2pt}%
    \setlength{\partopsep}{0pt}%
                \setlength{\itemsep}{2pt}%
    \setlength{\parsep}{0pt}%
    \setlength{\leftmargin}{2.5em}%
    \setlength{\labelwidth}{1.5em}%
    \setlength{\labelsep}{0.5em}%
    \setlength{\listparindent}{0pt}%
    \setlength{\itemindent}{0pt}%
  }%
}%
{\end{list}}
\def\CE{\mathcal{U}}
\def\emptyword{\emptyset}
\def\emptystate{\mathbf{\emptyset}}

\def\tx{\tilde x}
\def\ty{\tilde y}
\newcommand\gre{\mathbf{e}}
\def\maL{\mathcal L}
\def\maM{\mathcal M}
\def\maN{\mathcal N}

\begin{document}

\title{Loynes construction for the extended bipartite matching}

\author{
Pascal Moyal, Ana Bu$\check{\mbox{s}}$i\'c and Jean Mairesse
}


\maketitle

\begin{abstract}
We propose an explicit construction of the stationary state of Extended Bipartite Matching (EBM) models, as defined in (Bu$\check{\mbox{s}}$i\'c {\em et al.}, 2013). We use a Loynes-type backwards scheme similar in flavor to that in 
(Moyal {\em et al.}, 2017), allowing to show the existence and uniqueness of a bi-infinite 
perfect matching under various conditions, for a large class of matching policies and of bipartite matching structures. 
The key algebraic element of our construction is the sub-additivity of a suitable stochastic recursive representation of the model, 
satisfied under most usual matching policies. By doing so, we also derive stability conditions for the system under general stationary ergodic assumptions, subsuming 
the classical markovian settings. 
   
\medskip

\noindent
{\em Keywords: Matching model, graphs, Stochastic recursions, Coupling}
\medskip

\noindent
{\em 2000 Mathematics Subject Classification:} Primary 60J10; 60G10; Secondary 60K25; 05C38; 05C70; 90B15.
\end{abstract}


\section{Introduction}
\label{sec:intro}
Motivated by a wide range of applications, bipartite matching models can be thought of as a generalization of the usual skill-based queueing system in which customers and servers play symmetric roles: instead of being part of the 'hardware' of the system, the servers come and go exactly 
as the customers. Service times are not taken into account, as customers and servers only use the system as an interface to be matched together, and leave the system by couples, 
as soon as they form one. A bipartite graph named {\em matching graph}, specifies the possible matchings, i.e. the classes of servers and customers are represented by the nodes of the graph, the bipartition 
consists of the sets $\maC$ of "server" nodes and $\maS$ of "customer" nodes, and the existence of an edge between customer node $i$ and server node $j$ means that customers of class $i$ can be attended by servers of class $j$. 
As is easily seen, such models have a large variety of applications, from organ exchange programs \cite{BDPS11} to housing allocation \cite{TW08} or taxi platforms (see \cite{BC15,BC17}, in which the considered matching graphs are 
complete - however each match cab/passenger occurs with some specified probability).  

Under general bipartite structures, this class of models was formalized in \cite{CKW09} and then \cite{AW11}, under the name {\em stochastic bipartite matching} (BM) model: assume that 
arrivals occur by pairs in discrete time. At each time, a pair customer/server is drawn from a given probability measure $\mu:=\mu_{\maC} \otimes \mu_{\maS}$ on the set $\maC \times \maS$, independently of everything else. 
Possible matches are specified by a fixed bipartite matching graph, and a {\em matching policy} decides which match to perform at any given time, in case of multiple choices. 
The above seminal references addressed the stability problem of such models, under the First Come, First Served (FCFS) policy.  
A general condition on $\mu$ is obtained (eq. (\ref{eq:Ncond}) below) that guarantees the existence of a perfect FCFS matching in the long run, using regeneration points.  
Interestingly, this condition is also necessary and sufficient for complete resource pooling to hold for the fluid approximation of the corresponding skill-based service system under the FCFS-ALIS 
(allocate to longest idle server) policy (see \cite{AW14}), and appears as a stochastic analog of Hall's necessary and sufficient condition for the existence of a perfect matching on a given bipartite graph, see \cite{Hall35}. 
In \cite{ABMW17}, using an indirect reversibility argument, the stationary measure of the FCFS BM model is proven to have a remarkable product form under (\ref{eq:Ncond}). 
Also, any initially empty system couples from the past to the above distribution, as shown using a backwards scheme {\em \`a la} Loynes \cite{Loynes62} (Section 3 in \cite{ABMW17}).

An extension of the BM model, termed {\em extended} bipartite matching (EBM) model was proposed in \cite{BGM13}, allowing the probability measure $\mu$ to have an arbitrary support on $\maC \times \maS$, 
in a way that $\mu$ cannot necessarily be written as a tensor product of two measures on $\maC$ and $\maS$. 
Couples customer/server in the support of $\mu$ are represented by a secondary graph, termed {\em arrival graph} on the bipartition $\maC\cup \maS$. 
An extensive stability study of EBM models was undertaken in \cite{BGM13}, in particular providing sufficient stability conditions on $\mu$ (eq. (\ref{eq:Scond}) below), and proving that the 'Match the longest' policy is always stable under (\ref{eq:Ncond}) - we say in such case that the considered matching policy has a {\em maximal} stability region. 

To suit alternative areas of applications, such as dating websites, collaborative economic architectures and assemble-to-order systems, another variation on matching systems was proposed in \cite{MaiMoy16}: 
in the so-called stochastic {\em General matching} (GM) model, the matching graph is non-necessarily bipartite (hence there are no such things as {\em customers} and {\em servers}), and items enter the system one by one 
following a fixed probability distribution on the set of nodes. General stability results are given in \cite{MaiMoy16}, including a necessary stability condition related to (\ref{eq:Ncond}) (which cannot be satisfied 
whenever the matching graph is bipartite - thereby justifying the assumption of pairwise arrivals in that case, to make the system stabilizable), the maximality of 'Match the Longest' 
and the study of particular graphical structures. A variant of the GM model in continuous time was then proposed in \cite{MoyPer17}, showing using fluid (in)stability arguments, that aside for a particular class of matching graphs, there always exist a matching policy of the strict priority type that is not maximal, and that the min-cost 'Uniform' policy, consisting in choosing the class of match uniformly at random among the non empty neighboring classes, is never maximal. The maximality of FCFS for GM models, and the product form of the stationary measure, was then shown in \cite{MBM17}, 
together with the coupling to the stationary state whenever the latter exists, in the strong backwards sense of Borovkov and Foss \cite{Bor84,Foss92}, for a various class of matching policies including FCFS, LCFS, Match the Longest and 'Uniform'. Related models are studied in 
\cite{GurWa}, \cite{BM14} and \cite{NS16}, proposing optimization schemes for models in which the matching structures are particular (bipartite) graphs or hypergraphs, and the matching schemes are associated to a cost or a reward.  

Coupling-from-the past convergence schemes, such as Loynes's construction, are a crucial tool for the explicit pathwise construction of the steady state. These techniques form a central keystone for the qualitative comparison of discrete-event systems under general statistical assumptions (see e.g. Chapter 4 of \cite{BacBre02}) and perfect simulation of the stationary state (see \cite{PW96}). 
Motivated by this practical interest, the present work consists of a generalization of Loynes's construction in Section 3 of \cite{ABMW17} to:
\begin{itemize}
\item {EBM models}, 
\item {a wider class of matching policies} (in fact, to most usual policies that do not allow delaying any possible match, thanks to a particular 'block-wise' sub-additive property that is specified below), 
\item {stationary ergodic} - but not necessarily independent - {inputs},
\item {a wider class of initial conditions}.
\end{itemize} 
For doing so, we adapt to the present context, the coupling arguments developed in \cite{MBM17} for GM models, which mostly use Borovkov and Foss Theory of renovation \cite{Foss92,Foss94}. This paper is organized as follows: the EBM model is formally introduced in Section \ref{sec:model}. 
A key sub-additive property of the model under most common matching policies is shown in Section \ref{sec:subadd}. After introducing abstract notions that will prove useful in the proofs to come 
(namely, bi-separable graphs and erasing couples, respectively in Sections \ref{sec:bisep} and \ref{sec:erase}, we construct a backwards scheme, and then state and prove our main coupling results in Section \ref{sec:loynes}, 
making explicit the construction of perfect bi-infinite matchings under stability conditions, in sub-section \ref{subsec:matchings}.

\section{Formal definition of the models}
\label{sec:model}

\subsection{Notation}
We denote by $\R$ the set of real numbers, $\Z$ the set of integers, and by $\N$ the subset of
non-negative integers.   
For any finite set $A$, we let $|A|$ denote the cardinality of $A$, $\varsigma(A)$ be the set of permutations of $A$, and $A^*$ be the free monoid generated by $A$. 
Let $\emptyword$ denote the empty word of $A^*$. Words of $A^*$ will typically be denoted by bold symbols, and their letters in the corresponding regular symbol, e.g. 
$\mw=w_1...w_{|\mw|}$, where $|\mw|$ denotes the length of the word $\mw$. 
For any $\mw \in A$ and any $B \subset A$, set $|\mw|_B = \# \{i \mid w_i \in B\}$, the number
of occurrences in $\mw$ of letters from $B$. For $B = \{b\}$, we shorten
the notation to $|\mw|_b$. 
Furthermore, for any $\mw \in A^*$ we set $[\mw]:=(|\mw|_a)_{a\in A}$, the commutative image of $w$. 
For a word $\mw \in A^*$ of length $k$ and $i\in \{1,\dots, k\}$, we denote by
$\mw_{[i]}:= w_1\ldots w_{i-1} w_{i+1}\ldots w_k$, the sub-word of $\mw$
obtained by deleting $w_i$. For positive integers $i,j$ such that $i\le j$, the $i$-th vector of
the canonical basis of $\mathbb R^j$ is denoted by $\gre_i$. 

For two positive integers $a$ and $b$, the denote by $\|.\|$ the $\ell_1$-norm of
$\mathbb R^{a}\times \mathbb R^{b}$, {\em i.e.} for all $(x,y)
\in \mathbb R^{a}\times \mathbb R^{b}$,
\[\|(x,y) \| = \sum_{i=1}^{a} |x(i)|+\sum_{j=1}^{b} |y(j)|.\] 
Finally, the commutative image of a couple $(\mw,\mz) \in A^*\times B^*$ is the following couple of 
$\N^{|A|}\times \N^{|B|}$, 
\[\left[(\mw,\mz)\right]:=\left([\mw],[\mz]\right).\] 

\subsection{Extended bipartite matching}
\label{subsec:defEBM}
We call {\em Extended Bipartite Matching} (EBM), the model introduced in \cite{BGM13}, from which  
we keep the main terminology and notation. For clarity, let us recall hereafter the main definitions of Section 2 therein: 

\begin{definition}
\rm
We call a {\em bipartite matching structure} a quadruple $\maB:=(\maC,\maS,E,F)$, where 
\begin{itemi}
\item $\maC$ (which we identify with $\{1,2,...,|\maC|\}$) is the non-empty and finite set of customer classes; 
\item $\maS$ (identified with $\{1,2,...,|\maS|\}$) is the non-empty and finite set of server classes;
\item $E \subset \maC\times\maS$ is the set of possible matchings;
\item $F \subset \maC\times\maS$ is the set of possible arrivals.
\end{itemi}
\end{definition} 

Given a structure $\maB$, we consider that customers and servers of various classes arrive in the system by pairs, 
and let $\maC$ and $\maS$ denote the sets of customer and server classes, respectively. 
The set of possible incoming pairs is given by $F$, and the set $E$ defines the pairs that may depart
from the system, aka the possible \emph{matchings}. We say that
$\maH:=(\maC \cup \maS,F)$ is the \emph{arrival graph} and that $\maG:=(\maC \cup \maS,E)$ is the {\em matching graph} of the model. 
We assume without loss of generality that
\begin{itemi}
\item $\maG$ is connected;
\item $\maH$ has no isolated vertices.
\end{itemi} 
For a matching graph $\maG=(\maC \cup \maS,E)$, we denote by $\maC(s)$ the set of
customer classes that can be matched with an $s$-server, and by
$\maS(c)$ the set of server classes that can be matched with a
$c$-customer:
$$\maS(c) = \{s \in\maS \; : \; (c,s) \in E\}, \quad \maC(s) = \{c \in \maC \; : \; (c,s) \in E\}.$$
For any subsets $A \subset \maC$, and $B \subset\maS$, we define
$$\maS(A) = \cup_{c\in A}\maS(c), \quad \maC(B) = \cup_{s\in B}\maC(s).$$
Upon arrival of a new ordered pair $(c,s)\in F$, two situations may
occur: if neither $c$ nor $s$ match with the servers/customers
already present in the system, then $c$ and $s$ are simply added to
the buffer; if $c$, resp. $s$, can be matched then it departs the
system with its match, which leaves the buffer forever. If several matchings are possible for $c$,
resp. $s$, then it is the role of the matching policy to select one. 
To properly define matching policies that may depend on (possibly random) choices 
of matches, we will need to represent the orders of preferences of the customers and servers upon arrivals by elements of 
the two following finite sets, 
\begin{align*}
\mathbb S &= \varsigma(\maS(1)) \times ... \times \varsigma(\maS(|\maC|));\\
\mathbb C &= \varsigma(\maC(1)) \times ... \times \varsigma(\maC(|S|)),
\end{align*}
in other words, e.g. for any $\sigma:= \left(\sigma(1),...,\sigma\left(|\maC|\right)\right) \in \mathbb S$ and $c \in \maC$, 
$\sigma(c)$ is a permutation of the classes of servers that are compatible with $c$. Then, identifying $\maS(c)$ with $\llbracket 1,|\maS(c)| \rrbracket$ we denote 
for all $k \in \llbracket 1,\maS(c) \rrbracket$, by $\sigma(c)[k]$ the $k$-th neighboring class of $c$ according to $\sigma$. Similarly for any $\gamma \in \mathbb C$. 
Any such array of permutations $\sigma \in \mathbb S$ (resp. $\gamma \in \mathbb C$) is called {\em list of customer (resp. server) preferences},  
and if the entering couple is $(c,s) \in F$, $\sigma(c)$ and $\gamma(s)$ are respectively understood as the order of preference of the entering customer and the entering server, 
among the classes of their possible matches. 
Then, a matching policy will be formalized as an operator mapping the system state to the next one, provided that the classes of the entering 
couple are $(c,s)\in F$ and the orders of preferences of the classes are given by $(\sigma,\gamma)\in\mathbb S \times \mathbb C$. 
The matching policies we consider are presented in detail in Section \ref{subsec:pol}.

\begin{definition}
\rm
We call an {\em extended bipartite matching} (EBM) model, a bipartite matching structure $(\maC,\maS,E,F)$ together with a matching policy $\phi$, and a (finite or infinite)  
family of ordered quadruples $(c_n,s_n,\sigma_n,\gamma_n)_{n\in \maN} \in \left(F\times (\mathbb S \times \mathbb C)\right)^{\maN}$.  
The array $(c_n,s_n,\sigma_n,\gamma_n)_{n\in \maN}$ is then called {\em input} of the EBM model. 
\end{definition}

\medskip

Observe that two classes of systems studied in the literature can be seen as special cases of
EBM models:
\begin{itemize}
\item The Bipartite Matching (BM) model corresponding to the version introduced in \cite{calkapwei09} is an EBM with $F=\maC\times\maS$.
\item The General Matching (GM) model, as introduced in \cite{MaiMoy16}, having a general matching graph $(\maC,R)$ on the set of vertices $\maC$, is an EBM with $S= \tilde{\maC}$, a disjoint copy of $\maC$, $F=\{(c,\tilde{c}), c\in \maC\}$ and 
$(\maC\cup\maS, E)$ is the bipartite double cover of $(\maC,R)$. 
\end{itemize}

\paragraph{An associated graph.} We consider the directed graph $(\maC \cup \maS,A)$ where the set of arcs $A$ is defined by:
\begin{itemi}
\item $(c,s) \in A$ if and only if $(c,s) \in E$ and
\item $(s,c) \in A$ if and only if $(c,s) \in F$.
\end{itemi}
As shown in Theorem 4.1 of \cite{BGM13}, irreducibility of the natural Markov representation of the system is closely related to the strong connectedness of 
the directed graph $A$. Observe the following, 
\begin{lemma}
\label{lemma:strongconnect}
In the BM model, if $(\maC \cup \maS,E)$ is connected then the graph $(\maC \cup \maS,A)$ is strongly connected. In the GM model, if the reduced graph $(\maC,R)$ is connected, then $(\maC \cup \maS,A)$ is strongly connected.
\end{lemma}

\begin{proof}
First consider a BM model of connected matching graph $(\maC \cup \maS, E)$, and let $c \in \maC$ and $s \in\maS$. There exists an alternating path 
$c \v s_1 \v c_1 \v ... \v s_{p} \v c_p \v s$ connecting $c$ to $s$ in $(\maC \cup \maS, E)$. Thus, as $F=\maC \times\maS$, 
$(\maC \cup \maS, A)$ contains both the oriented path $c \to s_1 \to c_1 \to s_2 \to c_2 \to ... \to s_p \to c_p \to s$ 
and the oriented path $j \to c_p \to s_p \to c_{p-1} \to s_{p-1} \to ... \to c_1 \to s_1 \to c$. Thus the oriented graph $(\maC \cup \maS,A)$ is strongly connected. 

We now consider a GM model of connected reduced graph $(\maC,R)$. Let $i \in \maC$ and $\tilde j \in\maS=\tilde{\maC}.$ 
By the connectedness there exists a path $i \v k_1 \v ... \v k_p \v j$ between $i$ and $j$ in $(\maC,R)$. As ($c, \tilde c') \in E$ for any $c \v c' \in \maC$ and $(c,\tilde c) \in F$ for any $c \in \maC$, 
$(\maC \cup \maS, A)$ contains both paths $i \to \tilde k_1 \to k_1 \to \tilde k_2 \to k_2 \to ... \to \tilde k_p \to k_p \to \tilde j$ 
and $\tilde j \to j \to \tilde k_p \to k_p \to \tilde k_{p-1} \to ... \to k_1 \to \tilde i \to i$; so it is strongly connected. 
\end{proof}

\subsection{State spaces}
\label{subsec:statespace}
Let $\maN$ be a set of cardinality $N$, identified with $\llbracket 1,N \rrbracket$. Fix an EBM model of input $(c_n,s_n,\sigma_n,\gamma_n)_{n\in \maN}$. 
Let $\mc=c_1...c_N$, and likewise for $\ms$, $\msg$ and $\mga$. 
In this case we will call for short {\em admissible input} of the EBM model, the couple of words $(\mc,\ms)$. 
Then, for a given admissible policy (to be properly defined in sub-section \ref{subsec:pol}), there exists a unique {\em matching} of the admissible input $(\mc,\ms)$, that is, a bipartite graph having set of nodes 
$\left\{c_1,...,c_{N}\right\}\cup \left\{s_1,...,s_{N}\right\}$, whose edges represent the matchings of the corresponding customers and servers. 
This matching is denoted $M_\phi(\mc,\ms,\msg,\mga)$. In turn, the {\em buffer detail} of $M_\phi(\mc,\ms,\msg,\mga)$, denoted by $Q_\phi(\mc,\ms,\msg,\mga)$, 
is the couple of words 
\[Q_\phi(\mc,\ms,\msg,\mga):=\Bigl(C_\phi(\mc,\ms,\msg,\mga)\,,\,S_\phi(\mc,\ms,\msg,\mga)\Bigl) \in \maC^*\times \maS^*,\] 
such that $C_\phi(\mc,\ms,\msg,\mga)$ (resp., $S_\phi(\mc,\ms,\msg,\mga)$) is the sub-word of $\mc$ (resp., $\ms$) 
whose letters are the classes of the unmatched customers in $\mc$ (resp. of the unmatched servers in $\ms$), in order of arrivals. 
Observe that the definitions of $M_\phi(\mc,\ms,\msg,\mga)$ and $Q_\phi(\mc,\ms,\msg,\mga)$ can be extended to words $\mc$ and $\ms$ of different sizes, as follows: if $\mc$ is of length $N$ and $\ms$ is of length 
$M$ for $N \ne M$ (say e.g. $N > M$), then we consider that $N-M$ customers first enter the system alone, and then $M$ arrivals occur by couples; in other words for all 
$n \in \llbracket 0,M-1 \rrbracket$, the couple $\left(c_{N-n},s_{M-n}\right)$ enters the system contemporarily, with lists of preferences $\left(\sigma_{N-n},\gamma_{M-n}\right)$, 
and for all $n \in \llbracket 1,N-1 \rrbracket$, we let $c_{N-n}$ be the class of a single customer, with an (irrelevant) arbitrary list of preferences.  

Any admissible buffer detail belongs to the set 
\begin{equation}
\CE = \Bigl\{ (\mw,\mz) \in \maC^*\times \maS^* \; : \; 
     \forall  (i,j) \in E, \; |\mw|_i|\mz|_j = 0  \Bigr\}.\label{eq-ncss}
     \end{equation}
We will denote shortly by $\emptystate$ the state $(\emptyword,\emptyword) \in \CE$, representing the empty buffers. 
Observe that any state $(\mw,\mz)$ corresponding to the input $(c_i,\sigma_i)_{i\le N}$ and $(s_j,\gamma_j)_{j\le M}$ as above, 
is such that $|\mw|=|\mz|$ if $N=M$. However we will often work in greater generality, so any state of $\CE$ is admissible. 
We denote by $\CE_0$, the subset of admissible states having equal numbers of customers and servers, i.e. 
 \begin{equation}
\CE_0 = \Bigl\{ (\mw,\mz) \in \CE \; : \; |\mw|=|\mz|\Bigr\}.\label{eq-ncss}
     \end{equation}
Whenever the matching policy $\phi$ is such that the preference permutations are irrelevant, we will just drop this parameter from all notation, and write  
e.g. $M_\phi(\mc,\ms)$ and $Q_\phi(\mc,\ms)$ for the matching of $\mc$ and $\ms$ and for the buffer detail of that matching, respectively.

As will be seen below, depending on the matching policy we can also restrict the available information on the state of the system, to a vector only keeping track of 
the number of customers and servers of the various classes remaining unmatched after the matching of the finite words $\mc$ and $\ms$, that is, of the couple of 
commutative images of $C_\phi(\mc,\ms,\msg,\mga)$ and $S_\phi(\mc,\ms,\msg,\mga)$. In such cases this restricted state, which we will be called {\em class detail} of the system, takes values in the set  
\begin{equation}
\mathbb E = \Bigl\{(x,y)\in \N^{|\maC|}\times
\N^{|S|}\,:\,x(i)y(j)=0\mbox{ for any }(i,j)\in E\Bigl\}=\Bigl\{\left[(\mw,\mz)\right];\,(\mw,\mz) \in \CE\Bigl\}.\label{eq-css}
\end{equation} 

\subsection{Matching policies}
\label{subsec:pol}
We now formally describe the main matching policies we consider. 
In all cases, we make the following {\em buffer-first} assumption: an incoming couple $(c,s)$ is matched together if and only if
$(c,s) \in E$, $c$ found no match in the buffer of servers {\em and} $s$ found no match in the buffer of customers.  

\begin{definition}
\rm
A matching policy $\phi$ is said {\em admissible} if the choice of matches of an incoming couple $(c,s)$ 
depends {\em only} on the buffer detail and the couple of lists of preferences $(\sigma,\gamma)$.  
\end{definition} 
In other words, if a matching policy $\phi$ is admissible there exists a mapping $\odot_{\phi}: \CE \times (F \times (\mathbb S \times \mathbb C)) \rightarrow \CE$ such that, 
denoting by $(\mw,\mz)$ the buffer detail corresponding to a given input, and by $(\mw',\mz')$ the buffer detail if the latter input is augmented by the quadruple $(c,s,\sigma,\gamma)$, then 
$(\mw',\mz')$ and $(\mw,\mz)$ are connected by the relation 
\begin{equation}
\label{eq:defodot}
(\mw',\mz')= (\mw, \mz) \odot_{\phi} (c,s,\sigma,\gamma).
\end{equation}

\paragraph{First Come, First Served.} 
In First Come, First Served ({\sc fcfs}), the lists of preference are irrelevant, and are erased from all notation for short. 
The map $\odot_{\textsc{fcfs}}$ is then given for all $(\mw,\mz) \in \CE$ and all couples $(c,s)$ by 
$$
(\mw, \mz) \odot_{\textsc{fcfs}} (c,s) =
\left \{
\begin{array}{ll}
(\mw c, \mz s), & \textrm{if } \; |\mw|_{\maC(s)} = 0, \;|\mz|_{\maS(c)} = 0, \;
  (c,s) \not\in E \\
(\mw, \mz), & \textrm{if } \; |\mw|_{\maC(s)} = 0, \; |\mz|_{\maS(c)} = 0, \; (c,s)\in E\\
(\mw_{\left [\Phi(w,s)\right]}, \mz_{\left [\Psi(z,c)\right]}), & \textrm{if } \; |\mw|_{\maC(s)} \not= 0, \; |\mz|_{\maS(c)} \not= 0\\
(\mw_{\left [\Phi(w,s)\right]}c, z), & \textrm{if } \; |\mw|_{\maC(s)} \not= 0, \; |\mz|_{\maS(c)} = 0\\
(\mw, \mz_{\left [\Psi(z,c)\right]}s), & \textrm{if } \; |\mw|_{\maC(s)} = 0, \; |\mz|_{\maS(c)} \not= 0,
\end{array}
\right .
$$
with functions $\Phi$ and $\Psi$ as follows:
$$\Phi(\mw, s) = \arg\min \{w_k \in \maC(s)\}, \quad \Psi(\mz, c) = \arg\min \{z_k \in \maS(c)\}.$$

\paragraph{Last Come, First Served.}The lists of preferences are  again irrelevant; the map $\odot_{\textsc{lcfs}}$ is analog to $\odot_{\textsc{fcfs}}$, for 
$$\Phi(\mw,s) = \arg\max \{w_k \in \maC(s)\}, \quad \Psi(\mz, c) = \arg\max \{z_k \in \maS(c)\}.$$

\begin{definition}
\rm
A matching policy $\phi$ will be said {\em class-admissible} if it is fully characterized by 
two mappings $p_\phi$ and $q_\phi$ from $\N^{|S|} \times \maC \times \mathbb S$ to $S$ (resp., $\N^{|\maC|} \times\maS \times \mathbb C$ to $\maC$) such that $p_\phi(y,c,a)$ (resp. $q_\phi(x,s,b)$) 
determines the class of the match (if any) chosen by the entering $c$-customer (resp. $s$-server) under $\phi$ in a system of class detail $(x,y)$, for the lists of preferences $(\sigma,\gamma)$.  
\end{definition}


Let us define for any $(c,s) \in F$ and 
$(x,y) \in \mathbb E$,
\begin{align*}
\mathscr P(y,c) &=\Bigl\{j\in \maS(c)\,:\,y\left(j\right) > 0\Bigl\};\\
\maQ(x,s) &=\Bigl\{i\in \maC(s)\,:\,x\left(i\right) >
0\Bigl\},\label{eq:setP2}
\end{align*}
which represent the set of classes of available compatible servers
(resp. customers) with the entering customer $c$ (resp. server $s$),
if the class detail of the system is given by $(x,y)$. 
As is easily seen, for a class-admissible $\phi$, the arrival of $(c,s) \in F$ and the draw of $(\sigma,\gamma)$ from $\nu^\phi$ correspond to the following action on
the class detail of the system,
\begin{equation}
\label{eq:defccc}
(x, y) \ccc_{\phi} (c,s,\sigma,\gamma) = \left \{
\begin{array}{ll}
(x+\gre_c,y+\gre_s), &\mbox{ if }\maP(y,c)=\emptyset,\,
\maP(x,s)=\emptyset,\,(c,s)\not\in E,\\
(x,y), &\mbox{ if }\maP(y,c)=\emptyset,\, \maP(x,s)=\emptyset,\,(c,s)\in E,\\
(x,y+\gre_s-\gre_{p_\phi(y,c,a)}), &\mbox{ if }\maP(y,c)\ne
\emptyset,\,\maP(x,s)=\emptyset,\\
(x+\gre_c-\gre_{q_\phi(x,s,b)},y), &\mbox{ if }\maP(y,c)=\emptyset,\,\maP(x,s)\ne
\emptyset,\\
(x-\gre_{q_\phi(x,s,b)},y-\gre_{p_\phi(y,c,a)}), &\mbox{ if
}\maP(y,c)\ne \emptyset,\,\maP(x,s)\ne
\emptyset.
\end{array}
\right .
\end{equation}

\begin{remark}
\label{rem:equiv}
\rm
The same observation as in Remark 1 in \cite{MBM17} holds: to any class-admissible policy $\phi$ corresponds an admissible policy, e.g. by specifying that 
the rule of choice within class is FCFS (i.e. the customer/server chosen is the {\em oldest} one in line within its class). Then any class-admissible policy $\phi$ is admissible, i.e. the mapping $\ccc_\phi$ from 
$\mathbb E \times (F \times (\mathbb S \times \mathbb C))$ to $\mathbb E$ can be detailed into a map $\odot_{\phi}$ from 
$\CE \times (F \times (\mathbb S \times \mathbb C))$ to $\CE$, as in (\ref{eq:defodot}), such that for any buffer detail $(\mw,\mz)$ and any $(c,s,\sigma,\gamma)$, 
denoting again $(\mw',\mz')=(\mw,\mz)\odot_\phi (c,s,\sigma,\gamma)$ we have that 
\[\left([\mw'],[\mz']\right) = \left([\mw],[\mz]\right)\ccc_\phi (c,s,\sigma,\gamma).\] 
\end{remark}

\paragraph{Random policies.} 
Here the only information that is needed to determine the choice of matches for the incoming items, is whether their various compatible classes have an empty queue or not. 
Specifically, the considered customer/server investigates its compatible classes 
of servers/customers in the order induced by the lists of preferences upon arrival, until it finds one having a non-empty buffer, if any. The customer/server is then matched with a server/customer of the latter class. 
In other words, we set   
\begin{align*}
p_{\textsc{rand}}(y,c,\sigma) &=\sigma(c)[k],\mbox{ where }k=\min \Bigl\{i \in
[\llbracket 1,|\maS(c)| \rrbracket\,:\,\sigma(c)[i]\in \maP(y,c)\Bigl\};\\
q_{\textsc{rand}}(x,s,\gamma) &=\gamma(s)[\ell],\mbox{ where }\ell=\min \Bigl\{j \in
\llbracket 1,|\maC(s)| \rrbracket\,:\,b(s)[j]\in \maQ(x,s)\Bigl\}.
\end{align*}
We call such policies 'Random' (and denote {\sc rand}) since, as will be formalized in Section \ref{subsec:NcondScond}, the lists of preference may be random, drawn from a given distribution on $\mathbb S \times \mathbb C$. 
The particular case where these lists are deterministic corresponds to a (strict) priority policy.

\paragraph{Match the Longest.}
In the 'Match the Longest' policy ({\sc ml}), the newly arrived customer/server chooses
a server/customer of the compatible class that has the longest line (not including the other incoming item 
whenever it is compatible). Ties between classes having queues of the same length are broken using the list of preference at this time. 
Formally, set for all $(c,s)$
and $(x,y)$ such that $\mathcal P(y,c) \ne \emptyset$ and $\mathcal P(x,s) \ne \emptyset$, 
\begin{align*}
L(y,c) &=\max\left\{y(j)\,:\,j \in \maS(c)\right\}\,\quad\mbox{ and }\quad\,
\maL(y,c) =\left\{i\in \llbracket 1,\maS(c) \rrbracket\,:\,y\left(i\right)=L(y,c)\right\}\subset \maP(y,c),\\
M(x,s) &=\max\left\{x(i)\,:\,i \in \maC(s)\right\}\,\quad\mbox{ and }\quad\,
\maM(x,s) =\left\{i\in \llbracket 1,\maC(s) \rrbracket\,:\,x\left(i\right)=M(x,s)\right\}\subset \maQ(x,s).
\end{align*}
\begin{align*}
p_{\textsc{ml}}(y,c,\sigma) &=\sigma(c)[k],\mbox{ where }k=\min \Bigl\{i \in
\llbracket 1,|\maS(c)| \rrbracket:\,\sigma(c)[i]\in \maL(y,c)\Bigl\};\\
q_{\textsc{ml}}(x,s,\gamma) &=\gamma(s)[\ell],\mbox{ where }\ell=\min \Bigl\{j \in
\llbracket 1,|\maC(s)| \rrbracket:\,\gamma(s)[j]\in \maM(x,s)\Bigl\}.
\end{align*}

\paragraph{Match the Shortest.}
'Match the Longest' {\sc ms} is defined analogously to {\sc ml}, except that the shortest queue is chosen instead of
the longest.

\section{Sub-additivity}
\label{sec:subadd} 
In this section we show that, under most matching policies we have introduced we have introduced above, the EBM model satisfies a sub-additivity property that will prove crucial in the construction of a backwards scheme.  
\begin{definition}[Sub-additivity]
\label{def:subadd}
An admissible matching policy $\phi$ is said to be {\em sub-additive} if, 
for all $\mc',\mc''\in \maC^*$, $\ms',\ms''\in \maS^*$, $\msg',\msg'' \in \mathbb S^*$ and $\mga',\mga'' \in \mathbb C^*$ such that $|\mc'|=|\ms'|=|\msg'|=|\mga'|$ and $|\mc''|=|\ms''|=|\msg''|=|\mga''|$, 
we have that 
\begin{align*}
\left|C_\phi\left(\mc'\mc'',\ms'\ms'',\msg'\msg'',\mga'\mga''\right)\right| &\leq \left|C_\phi\left(\mc',\ms',\msg',\mga'\right)\right| + \left|C_\phi\left(\mc'',\ms'',\msg'',\mga''\right)\right|;\\ 
\left|S_\phi\left(\mc'\mc'',\ms'\ms'',\msg'\msg'',\mga'\mga''\right)\right| &\leq \left|S_\phi\left(\mc',\ms',\msg',\mga'\right)\right| + \left|S_\phi\left(\mc'',\ms'',\msg'',\mga''\right)\right|.
\end{align*}
\end{definition}
It is proven in Lemma 4 of \cite{ABMW17} that FCFS is sub-additive for the BM model; as the sub-additivity is purely algebraic and does not depend on the statistics of the input, 
this also clearly holds true for the EBM model. Also, Proposition 3 in 
\cite{MBM17} shows that a similar sub-additive property is satisfied for the GM model for various matching policies. 
The result below generalizes the former result to a larger class of matching policies, thereby providing an analog to the latter result in the bipartite case,  
\begin{proposition}
\label{prop:sub}
The matching policies {\sc fcfs}, {\sc lcfs}, {\sc rand} and {\sc ml} are sub-additive. 
\end{proposition}
Before proving Proposition \ref{prop:sub} let us observe that, similarly to Example 3 of \cite{MBM17}, 
\begin{ex}[{\sc ms} is not sub-additive]
\label{ex:MS}
\rm
Take as a matching graph, the graph of Figure \ref{Fig:NN}, and the arrival scenario depicted in Figure \ref{Fig:MS}.
\begin{figure}[h!]
\begin{center}
\begin{tikzpicture}
\draw[-] (1,4) -- (1,1);
\draw[-] (4,4) -- (4,1);
\draw[-] (4,4) -- (1,4);
\draw[-] (4,1) -- (1,1);
\draw[-] (1.5,3) -- (1.5,2); 
\draw[-] (1.5,3) -- (2.5,2); 
\draw[-] (2.5,3) -- (2.5,2);
\draw[-] (2.5,3) -- (3.5,2);
\draw[-] (3.5,2)-- (3.5,3); 
\fill (1.5,3) circle (2.5pt) node[above] {\small{1}} ;
\fill (2.5,3) circle (2.5pt) node[above] {\small{2}} ;
\fill (3.5,3) circle (2.5pt) node[above] {\small{3}} ;
\fill (1.5,2) circle (2.5pt) node[below] {\small{$\bar 1$}} ;
\fill (2.5,2) circle (2.5pt) node[below] {\small{$\bar 2$}} ;
\fill (3.5,2) circle (2.5pt) node[below] {\small{$\bar 3$}} ;
\end{tikzpicture}
\caption[smallcaption]{The 'NN' graph.} \label{Fig:NN}
\end{center}
\end{figure}
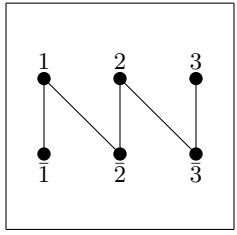

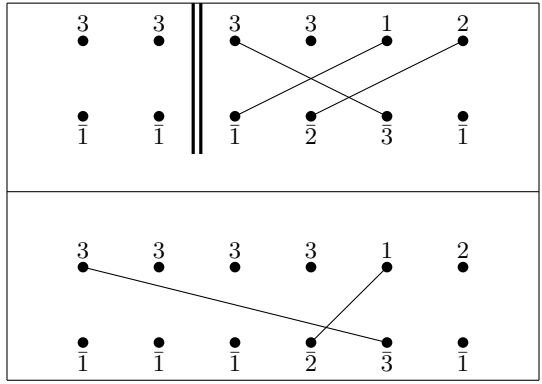
\begin{figure}[h!]
\begin{center}
\begin{tikzpicture}
\draw[-] (-1,0.5) -- (6,0.5);
\draw[-] (-1,0.5) -- (-1,-4.5);
\draw[-] (-1,-2) -- (6,-2);
\draw[-] (-1,-4.5) -- (6,-4.5);
\draw[-] (6,-4.5) -- (6,0.5);
\fill (0,0) circle (2pt) node[above] {\small{3}} ;
\fill (1,0) circle (2pt) node[above] {\small{3}} ;
\draw[-,very thick] (1.45,0.5) -- (1.45,-1.5);
\draw[-,very thick] (1.55,0.5) -- (1.55,-1.5);
\fill (2,0) circle (2pt) node[above] {\small{3}} ;
\draw[-] (2,0)-- (4,-1);
\fill (3,0) circle (2pt) node[above] {\small{3}} ;
\fill (4,0) circle (2pt) node[above] {\small{1}} ;
\draw[-] (4,0)-- (2,-1);
\fill (5,0) circle (2pt) node[above] {\small{2}} ;
\draw[-] (5,0)-- (3,-1);
\fill (0,-1) circle (2pt) node[below] {\small{$\bar 1$}} ;
\fill (1,-1) circle (2pt) node[below] {\small{$\bar 1$}} ;
\fill (2,-1) circle (2pt) node[below] {\small{$\bar 1$}} ;
\fill (3,-1) circle (2pt) node[below] {\small{$\bar 2$}} ;
\fill (4,-1) circle (2pt) node[below] {\small{$\bar 3$}} ;
\fill (5,-1) circle (2pt) node[below] {\small{$\bar 1$}} ;  
\fill (0,-3) circle (2pt) node[above] {\small{3}} ;
\draw[-] (0,-3)-- (4,-4);
\fill (1,-3) circle (2pt) node[above] {\small{3}} ;
\fill (2,-3) circle (2pt) node[above] {\small{3}} ;
\fill (3,-3) circle (2pt) node[above] {\small{3}} ;
\fill (4,-3) circle (2pt) node[above] {\small{1}} ;
\draw[-] (4,-3)-- (3,-4);
\fill (5,-3) circle (2pt) node[above] {\small{2}} ;
\fill (0,-4) circle (2pt) node[below] {\small{$\bar 1$}} ;
\fill (1,-4) circle (2pt) node[below] {\small{$\bar 1$}} ;
\fill (2,-4) circle (2pt) node[below] {\small{$\bar 1$}} ;
\fill (3,-4) circle (2pt) node[below] {\small{$\bar 2$}} ;
\fill (4,-4) circle (2pt) node[below] {\small{$\bar 3$}} ;
\fill (5,-4) circle (2pt) node[below] {\small{$\bar 1$}} ;
\end{tikzpicture}
\caption[smallcaption]{A non sub-additive {\sc ms} matching on the 'NN' graph.} \label{Fig:MS}
\end{center}
\end{figure}
As illustrated in Figure \ref{Fig:MS} we have 
\begin{align*}
\left|C_{\textsc{ms}}(333312,\bar 1\bar 1\bar 1\bar 2 \bar 3 \bar 1)\right| &= |3332| > |33| + |3| = \left|C_{\textsc{ms}}(33,\bar 1\bar 1)\right| + \left|C_{\textsc{ms}}(3312,\bar 1\bar 2 \bar 3 \bar 1)\right|;\\
\left|S_{\textsc{ms}}(333312,\bar 1\bar 1\bar 1\bar 2 \bar 3 \bar 1)\right| &= |\bar 1\bar 1\bar 1\bar 1| > |\bar 1\bar 1| + |\bar 1| = \left|S_{\textsc{ms}}(33,\bar 1\bar 1)\right| + \left|S_{\textsc{ms}}(3312,\bar 1\bar 2 \bar 3 \bar 1)\right|.
\end{align*}
\end{ex} 

The remainder of this section is devoted to the proof of Proposition \ref{prop:sub}. The result is already known for $\phi=\textsc{fcfs}$ (this is Lemma 4 in \cite{ABMW17}). 
The proof for non-expansive matching policies (including {\sc rand} and {\sc ml}) is given in Sub-section \ref{subsec:nonexp}, and that for {\sc lcfs} in Sub-section \ref{subsec:lcfs}.  

\subsection{Non-expansiveness} 
\label{subsec:nonexp}
The non-expansiveness of the class detail, as defined in \cite{MoyPer17} for GM models, is a 
Lipschitz property of the driving map of the recursion that has interest for constructing stochastic approximations of the model 
under consideration (see Section 6 in \cite{MoyPer17}). As we show in Proposition \ref{prop:nonexp}, this property is in fact stronger than 
sub-additivity for the buffer detail sequence,  
\begin{definition}
A class-admissible policy $\phi$ is said {\em non-expansive} if 
for any $(x,y)$ and $(x',y')$ in $\mathbb E$, and for any $(c,s) \in F$ and any $(\sigma,\gamma) \in \mathbb S\times \mathbb C$ 
that can be drawn by $(\nu_\phi,\rho_\phi)$, 
\begin{equation}
\label{eq:defnonexp1}
\|(x',y')\ccc_{\phi}(c,s,\sigma,\gamma) - (x,y)\ccc_{\phi}(c,s,\sigma,\gamma)\| \le \|(x',y')-(x,y)\|.
\end{equation}
\end{definition} 

The following results, transposing Lemma 7 in \cite{MoyPer17} and Propositions 4 and 5 in \cite{MBM17} to the EBM model, are proven in Appendix, 

\begin{proposition}
\label{prop:nonexp1} 
Any {\sc rand} matching policy is non-expansive. 
\end{proposition}

\begin{proposition}
\label{prop:nonexp2} 
{\sc ml} is non-expansive. 
\end{proposition}

Likewise section 4.2.2. in \cite{MBM17}, 
\begin{proposition} 
\label{prop:nonexp}
Any non-expansive matching policy is sub-additive. 
\end{proposition}
\begin{proof} 
Fix a non-expansive matching policy $\phi$. Keeping the 
notations of Definition \ref{def:subadd}, let us define the two
sequences $\{(x_n,y_n)\}$ and $\{(x'_n,y'_n)\}$ to be the class
details of the system at arrival times, starting respectively from
an empty system and from a system of buffer detail $(\mw',\mz')$, and
having a common input $\left(c'',s'',\sigma'',\gamma''\right)$. 
In other words, we set 
\[\left\{\begin{array}{ll}
(x_0,y_0) &= \left(\mathbf 0,\mathbf 0\right);\\
(x'_0,y'_0) &= \left[\left(C_\phi(\mc',\ms',\msg',\mga')\,,\,S_\phi(\mc',\ms',\msg',\mga')\right)\right]\\
\end{array}\right.\]
and by induction,
\[\left\{\begin{array}{ll}
\left(x_{n+1},y_{n+1}\right) &= \left(x_{n},y_{n}\right)\ccc_{\phi} \left(c''_{n+1},s''_{n+1},\sigma''_{n+1},\gamma''_{n+1}\right),\,n\in\left\{0,\dots,|\mc''|-1\right\};\\
\left(x'_{n+1},y'_{n+1}\right)&=\left(x'_{n},y'_{n}\right)\ccc_{\phi}\left(c''_{n+1},s''_{n+1},\sigma''_{n+1},\gamma''_{n+1}\right),\,n\in\left\{0,\dots,|\mc''|-1\right\}.
\end{array}\right.\]
Applying (\ref{eq:defnonexp1}) at all $n$, we obtain by induction that for all $n
\in \left\{0,\dots,|\mc''|\right\}$,
\begin{equation}
 \|(x'_n,y'_n) -(x_n,y_n)\| \le \|(x'_0,y'_0)-(x_0,y_0)\|= |w'|+|z'|.\label{eq:nonexprec}
\end{equation}
Now observe that by construction, $\left(x_{|\mc''|},y_{|\mc''|}\right)=\left(\left[C_\phi(\mc',\ms',\sigma',\gamma')\right]\,,\,\left[S_\phi(\mc',\ms',\sigma',\gamma')\right]\right)$ which, 
together with (\ref{eq:nonexprec}), implies that
\begin{align}
|\mw|+|\mz| &= \left\|\left(x'_{|\mc''|},y'_{|\mc''|}\right)\right\|\nonumber\\
        &\le \left\|\left(x'_{|\mc''|},y'_{|\mc''|}\right) -\left(x_{|\mc''|},y_{|\mc''|}\right)\right\|+\left\|\left(x_{|\mc''|},y_{|\mc''|}\right)\right\|\nonumber\\
        &\le \left|C_\phi(\mc',\ms',\sigma',\gamma')\right| + \left|S_\phi(\mc',\ms',\sigma',\gamma')\right| + \left|C_\phi(\mc'',\ms'',\sigma'',\gamma'')\right| 
         + \left|S_\phi(\mc'',\ms'',\sigma'',\gamma'')\right|.\label{eq:Auster1}
        \end{align}
Remember that the couples of words $\mc'$ and $\ms'$, and $\mc''$ and
$\ms''$, are not necessarily of the same size, and let us
denote
\[q'= |\ms'|-|\mc'|\quad\mbox{ and }\quad q''= |\ms''|-|\mc''|.\]
By the very definition of a matching of $\mc'$ and $\ms'$ we have that
\[\left|S_\phi(\mc',\ms',\msg',\mga')\right| - \left|C_\phi(\mc',\ms',\msg',\mga')\right| = |\ms'|-|\mc'| = q'\] and likewise, that
$\left|S_\phi(\mc'',\ms'',\msg'',\mga'')\right|-\left|C_\phi(\mc'',\ms'',\msg'',\mga'')\right|=q''$ and $\left|S_\phi(\mc,\ms,\msg,\mga)\right|-\left|C_\phi(\mc,\ms,\msg,\mga)\right|=|v|-|u|=q'+q''$. All in all, we
obtain with (\ref{eq:Auster1}) that
\[\left\{\begin{array}{ll}
2\left|C_\phi(\mc,\ms,\msg,\mga)\right|+q'+q'' &\le 2\left|C_\phi(\mc',\ms',\msg',\mga')\right|+q' + 2\left|C_\phi(\mc'',\ms'',\msg'',\mga'')\right|+q'';\\
2\left|S_\phi(\mc,\ms,\msg,\mga)\right|-(q'+q'') &\le 2\left|S_\phi(\mc',\ms',\msg',\mga')\right|-q' + 2\left|S_\phi(\mc'',\ms'',\msg'',\mga'')\right|-q''. 
\end{array}\right.\]
We conclude by adding up the two above inequalities. 
\end{proof}

\subsection{Proof of Proposition \ref{prop:sub} for {\sc lcfs}}
\label{subsec:lcfs}
Likewise {\sc fcfs}, the policy {\sc lcfs} does not satisfy a non-expansiveness property similar to 
(\ref{eq:defnonexp1}). We resort to a direct argument, similar to that of Section 4.2.3 in \cite{MBM17}.  
We drop for short the dependence on $(\sigma,\gamma)$ in the notations $\Mlcfs(.)$, $\Qlcfs(.)$ $\Clcfs(.)$ and $\Slcfs$, as the {\sc lcfs} matchings do not depend on any list of preferences. 

{\bf Step I:} At first, set $|\mc'|=1$ and $|\ms'|=0$.  
We need to show that $|\Clcfs(\mc,\ms)| \le |\Clcfs(\mc'',\ms'')|+1$ and $|\Slcfs(\mc,\ms)| \le |\Slcfs(\mc'',\ms'')|$. There are three different cases:  
\begin{itemize}
\item[(a)] If $c'_1$ is unmatched in $\Mlcfs(\mc,\ms)$, then $c'_1$ is incompatible with $\ms''_1$, otherwise the two items would have been matched. 
In turn, if follows from the definition of {\sc lcfs} that the presence of $c'_1$ does not influence the choice of match of any 
server $s''_j$ that is matched in $\Mlcfs(\mc'',\ms'')$, even though $(c'_1,s''_j) \in E$. So $|\Clcfs(\mc,\ms)|=|\Clcfs(\mc'',\ms'')|+1$ and $|\Slcfs(\mc,\ms)|= |\Slcfs(\mc'',\ms'')|$. 
\item[(b)] Whenever $c'_1$ is matched in $\Mlcfs(\mc,\ms)$ with a server $s''_{j_1}$ that was unmatched in $\Mlcfs(\mc'',\ms''),$ any 
matched customer $c''_i$ in $\Mlcfs(\mc'',\ms'')$ that is compatible with $s''_{j_1}$ has found in $\ms''$ a more recent compatible server $s''_j$. 
The matching of $c''_i$ with $s''_j$ still occurs in $\Mlcfs(\mc,\ms)$. Thus, as above the matching induced in $\Mlcfs(\mc,\ms)$ by the nodes of $\mc''$ is not affected by the 
match $(c_1,c''_{j_1})$, so $|\Clcfs(\mc,\ms)|=|\Clcfs(\mc'',\ms'')|$ and $|\Slcfs(\mc,\ms)|=|\Slcfs(\mc'',\ms'')|-1$.   
\item[(c)] Suppose now that $c'_1$ is matched with a server $s''_{j_1}$ that was matched in $\Mlcfs(\mc'',\ms'')$ to some customer $c''_{i_1}$. 
This occurs if and only if $c'_1 \v s''_{j_1}$ and $i_1 \ge j_1$, otherwise in {\sc lcfs}, $s''_{j_1}$ would prioritize $c''_{i_1}$ over $c'_1$ upon arrival.  
We thus need to search a new match for $c''_{i_1}$. Either there is none and we stop, or we find a match, say $s''_{j_2}$.
The new pair $(c''_{i_1}, s''_{j_2})$ potentially broke an old pair $(u''_{i_2}, v''_{j_2})$. We continue until either $u''_{i_k}$ cannot find a new match or $v''_{j_k}$ was not previously matched.
In the first case, we end up with $|\Clcfs(\mc'',\ms'')|+1$ unmatched customers and $|\Slcfs(\mc'',\ms'')|$ unmatched servers; and in the second case, with $|\Clcfs(\mc'',\ms'')|$ unmatched customers and 
$|\Slcfs(\mc'',\ms'')|-1$ unmatched servers.
\end{itemize} 
The case with $|\mc'|=0$ and $|\ms'|=1$ is symmetrical: then $\left|\Clcfs(\mc,\ms)\right| \le \left|\Clcfs(\mc'',\ms'')\right|$ and 
$\left|\Slcfs(\mc,\ms)\right| \le \left|\Slcfs(\mc'',\ms'')\right|+1$.

\bigskip

{\bf Step II:} Consider now arbitrary finite words $\mc'$ and $\ms'$. Note that if $(c'_i, s'_j) \in \Mlcfs(\mc',\ms')$ then
$(c'_i,s'_j) \in \Mlcfs(\mc,\ms)$, as is the case for any admissible policy. Thus we have 
\begin{equation}
\label{eq:tahri0}
\left\{\begin{array}{ll}
\Clcfs(\mc,\ms) &=\Clcfs\left(\Clcfs(\mc',\ms')\mc''\,,\,\Slcfs(\mc',\ms')\ms''\right);\\
\Slcfs(\mc,\ms) &=\Slcfs\left(\Clcfs(\mc',\ms')\mc''\,,\,\Slcfs(\mc',\ms')\ms''\right).
\end{array}\right.
\end{equation}
We will consider one by one the customers in $\Clcfs(\mc',\ms')$ and servers in $\Slcfs(\mc',\ms')$, starting from the right to the left. 
The order between customers and servers does not matter; we chose to consider customers first and then servers. 
Let for all $1 \leq i \leq \left|\Clcfs(\mc',\ms')\right|$ (resp. $1 \leq i \leq \left|\Slcfs(\mc',\ms')\right|$), $C^i$ (resp., $S^i$) 
be the suffix of length $i$ of $\Clcfs(\mc',\ms')$ (resp., $\Slcfs(\mc',\ms')$). 
By Step I, for $1 \leq i \leq \left|\Clcfs(\mc',\ms')\right|-1$ we have that 
\begin{equation}
\label{eq:tahri1}
\left\{\begin{array}{ll}
\left|\Clcfs\left(C^{i+1}\mc'', \ms''\right)\right| &\leq 1+\left|\Clcfs\left(C^{i}\mc'', \ms''\right)\right| \\
\left|\Slcfs\left(C^{i+1}\mc'', \ms''\right)\right| &\leq \left|\Slcfs\left(C^{i}\mc'', \ms''\right)\right|. 
\end{array}\right.
\end{equation}
Similarly, by considering then the servers from right to left we obtain that 
for $1 \leq i \leq \left|\Slcfs(\mc',\ms')\right|-1$,  
\begin{equation}
\label{eq:tahri2}
\left\{\begin{array}{ll}
\left|\Clcfs\left(\Clcfs(\mc',\ms')\mc'', S^{i+1}\ms''\right)\right| &\leq \left|\Clcfs\left(\Clcfs(\mc',\ms')\mc'',S^{i}\ms''\right)\right|\\ 
\left|\Slcfs\left(\Clcfs(\mc',\ms')\mc'', S^{i+1}\ms''\right)\right| &\leq 1+\left|\Slcfs\left(\Clcfs(\mc',\ms')\mc'',S^{i}\ms''\right)\right|. 
\end{array}\right.
\end{equation}
Applying (\ref{eq:tahri1}) by induction, and then (\ref{eq:tahri2}) by induction we obtain that 
\begin{equation*}
\left\{\begin{array}{ll}
\left|\Clcfs\left(\Clcfs(\mc',\ms')\mc'', \Slcfs(\mc',\ms')\ms''\right)\right| &\leq \left|\Clcfs(\mc',\ms')\right|+\left|\Clcfs\left(\mc'', \ms''\right)\right|\\
\left|\Slcfs\left(\Clcfs(\mc',\ms')\mc'', \Slcfs(\mc',\ms')\ms''\right)\right| &\leq \left|\Slcfs(\mc',\ms')\right| + \left|\Slcfs\left(\mc'', \ms''\right)\right|,
\end{array}\right.
\end{equation*}
which, together with (\ref{eq:tahri0}), concludes the proof for {\sc lcfs}. 


\section{Bi-separable graphs}
\label{sec:bisep}
We introduce a class of bipartite matching graphs which play an important role in the stability study that follows: the bi-separable graphs, 
which can be seen as the analog, for bipartite graphs, of the separable graphs introduced in \cite{MaiMoy16}. 

\begin{definition}
Let $\maG=(\maC \cup \maS,E)$ be a matching graph. An {\em independent set} of $\maG$ is a non-empty bipartite set $A \cup B$, where $A \subset \maC$ and $B \subset\maS$, that is such that 
$A \times B \subset (\maC\times\maS) \setminus E.$
Denote for any independent set $I:=A\cup B$ of $\maG$, 
\begin{equation}
\label{eq:defCSnot}
\maC_\circ (I) =\maC \setminus \left(A \cup \maC(B)\right)\quad\quad\mbox{ and }\quad\quad
\maS_\circ (I) =\maS \setminus \left(B \cup \maS(A)\right).
\end{equation}
The independent set $I$ is then said {\em maximal} whenever $\maC_\circ\left(I\right)=\emptyset\mbox{ and }\maS_\circ\left(I\right)=\emptyset.$
\end{definition}

\begin{definition}
A bipartite matching graph $\maG=\left(\maC \cup \maS,E\right)$ is said 
{\em bi-separable} of order $p$, $p\ge 2$, whenever there exists a partition of $\maC\cup\maS$ into $p$ independent sets
$I_1=A_1 \cup B_1,..,I_p=A_p \cup B_p$ of $\maG$ such that any $I_i$ that is not maximal (if any) is such that $A_i=\emptyset$ or $B_i=\emptyset$. In other words, the complement bipartite graph $\overline{\maG}$ of $\maG$ has $p$ disjoint bipartite-connected components $I_1,..,I_p$.
\end{definition}

\noindent Clearly,
\begin{proposition}
If the bipartite graph $\maG=\left(\maC \cup \maS,E\right)$ is bi-separable, then any bipartite graph
$\hat{\maG}=\left(\maC \cup \maS, \hat E\right)$ for $E \subset \hat E$, 
is bi-separable.
\end{proposition}

\noindent Also,
\begin{lemma}
\label{lemma:separable}
Let $\maG=\left(\maC \cup \maS,E\right)$ is bi-separable of order $p$ and maximal independent sets $I_1=(A_1,B_1),...,I_p=(A_p,B_p)$. Then, for any $i \in \llbracket 1,p \rrbracket$ and 
any $c,c' \in A_i$, $s,s' \in\maS$, we have \[\maS(c)=\maS(c')=S\setminus B_i\quad\quad\mbox{ and }\quad\quad \maC(s)=\maC(s')=\maC\setminus A_i.\]
\end{lemma}
\begin{proof}
Let $c,c' \in A_i$. Suppose that there exists $\ell \in\maS$ such that $(c,\ell) \in E$ but $(c',\ell) \not\in E$. 
Then, $\ell \in \maS(A_i),$ and thus $\ell \not\in B_i$. But $(c',\ell) \in \bar E$, so there exists $j \in \llbracket 1,p \rrbracket$ such that $(c',\ell) \in A_j\times B_j$.
This implies in turn that $j=i$, an absurdity since $\ell \not\in B_i.$
\end{proof}

Examples of bi-separable graphs are represented in Figure \ref{Fig:sep}. 

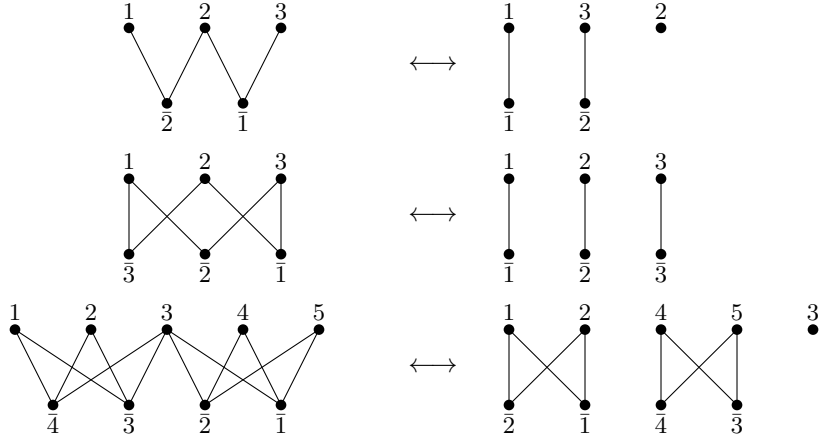
\begin{figure}[h!]
\begin{center}
\begin{tikzpicture}
\draw[-] (1,3) -- (1.5,2);
\draw[-] (2,3) -- (1.5,2);
\draw[-] (2,3) -- (2.5,2);
\draw[-] (3,3) -- (2.5,2);
\fill (1,3) circle (2pt) node[above] {\small{1}} ;
\fill (1.5,2) circle (2pt) node[below] {\small{$\bar 2$}} ;
\fill (2,3) circle (2pt) node[above] {\small{2}} ;
\fill (3,3) circle (2pt) node[above] {\small{3}} ;
\fill (2.5,2) circle (2pt) node[below] {\small{$\bar 1$}} ;
\fill (5,2.5) node[] {$\longleftrightarrow$} ;
\draw[-] (6,3) -- (6,2);
\draw[-] (7,3) -- (7,2);
\fill (6,3) circle (2pt) node[above] {\small{1}} ;
\fill (6,2) circle (2pt) node[below] {\small{$\bar 1$}} ;
\fill (7,3) circle (2pt) node[above] {\small{3}} ;
\fill (7,2) circle (2pt) node[below] {\small{$\bar 2$}} ;
\fill (8,3) circle (2pt) node[above] {\small{2}} ;
\draw[-] (1,1) -- (1,0);
\draw[-] (1,1) -- (2,0);
\draw[-] (2,1) -- (1,0);
\draw[-] (2,1) -- (3,0);
\draw[-] (3,1) -- (2,0);
\draw[-] (3,1) -- (3,0);
\fill (1,1) circle (2pt) node[above] {\small{1}} ;
\fill (1,0) circle (2pt) node[below] {\small{$\bar 3$}} ;
\fill (2,1) circle (2pt) node[above] {\small{2}} ;
\fill (2,0) circle (2pt) node[below] {\small{$\bar 2$}} ;
\fill (3,1) circle (2pt) node[above] {\small{3}} ;
\fill (3,0) circle (2pt) node[below] {\small{$\bar 1$}} ;
\fill (5,0.5) node[] {$\longleftrightarrow$} ;
\draw[-] (6,1) -- (6,0);
\draw[-] (7,1) -- (7,0);
\draw[-] (8,1) -- (8,0);
\fill (6,1) circle (2pt) node[above] {\small{1}} ;
\fill (6,0) circle (2pt) node[below] {\small{$\bar 1$}} ;
\fill (7,1) circle (2pt) node[above] {\small{2}} ;
\fill (7,0) circle (2pt) node[below] {\small{$\bar 2$}} ;
\fill (8,1) circle (2pt) node[above] {\small{3}} ;
\fill (8,0) circle (2pt) node[below] {\small{$\bar 3$}} ;
\draw[-] (-0.5,-1) -- (0,-2);
\draw[-] (-0.5,-1) -- (1,-2);
\draw[-] (0.5,-1) -- (0,-2);
\draw[-] (0.5,-1) -- (1,-2);
\draw[-] (1.5,-1) -- (1,-2);
\draw[-] (1.5,-1) -- (2,-2);
\draw[-] (1.5,-1) -- (0,-2);
\draw[-] (1.5,-1) -- (3,-2);
\draw[-] (2.5,-1) -- (2,-2);
\draw[-] (2.5,-1) -- (3,-2);
\draw[-] (3.5,-1) -- (2,-2);
\draw[-] (3.5,-1) -- (3,-2);
\fill (-0.5,-1) circle (2pt) node[above] {\small{1}} ;
\fill (0.5,-1) circle (2pt) node[above] {\small{2}} ;
\fill (1.5,-1) circle (2pt) node[above] {\small{3}} ;
\fill (2.5,-1) circle (2pt) node[above] {\small{4}} ;
\fill (3.5,-1) circle (2pt) node[above] {\small{5}} ;
\fill (0,-2) circle (2pt) node[below] {\small{$\bar 4$}} ;
\fill (1,-2) circle (2pt) node[below] {\small{$\bar 3$}} ;
\fill (2,-2) circle (2pt) node[below] {\small{$\bar 2$}} ;
\fill (3,-2) circle (2pt) node[below] {\small{$\bar 1$}} ;
\fill (5,-1.5) node[]{$\longleftrightarrow$};
\draw[-] (6,-1) -- (6,-2);
\draw[-] (7,-1) -- (7,-2);
\draw[-] (6,-1) -- (7,-2);
\draw[-] (7,-1) -- (6,-2);
\draw[-] (8,-1) -- (8,-2);
\draw[-] (9,-1) -- (9,-2);
\draw[-] (8,-1) -- (9,-2);
\draw[-] (9,-1) -- (8,-2);
\fill (6,-1) circle (2pt) node[above] {\small{1}} ;
\fill (7,-1) circle (2pt) node[above] {\small{2}} ;
\fill (8,-1) circle (2pt) node[above] {\small{4}} ;
\fill (9,-1) circle (2pt) node[above] {\small{5}} ;
\fill (6,-2) circle (2pt) node[below] {\small{$\bar 2$}} ;
\fill (7,-2) circle (2pt) node[below] {\small{$\bar 1$}} ;
\fill (8,-2) circle (2pt) node[below] {\small{$\bar 4$}} ;
\fill (9,-2) circle (2pt) node[below] {\small{$\bar 3$}} ;
\fill (10,-1) circle (2pt) node[above] {\small{3}} ;
\end{tikzpicture}
\caption[smallcaption]{Three bi-separable graphs of order 3, together with their bipartite complement graphs.}
\label{Fig:sep}
\end{center}
\end{figure}

\section{(Strong) erasing couples}
\label{sec:erase}
{\em Erasing couples} and {\em strong erasing couples} generalize to EBM models, the definitions of erasing words and strong erasing words for GM models 
(Definitions 4 and 6 in \cite{MBM17}). They will play an important role in the construction below. 

\subsection{Definitions}
\label{subsec:deferasing}

\begin{definition}
Let $\maB =(\maC,\maS,E,F)$ be a bipartite matching structure, $\phi$ be an admissible matching policy and $(\mw,\mz) \in \CE_0$.   
An admissible input $(\mc,\ms)$ is said to be an {\em erasing couple} of $(\mw,\mz)$ for $(\maB,\phi)$ if for any words $\msg,\msg'\in \mathbb C^*$ and $\mga,\mga'\in \mathbb S^*$ 
such that $|\msg|=|\mga|=|\mw|$ and $|\msg'|=|\mga'|=|\mc|$, we have that 
\begin{equation}
\label{eq:deferase}
Q_\phi\left(\mc,\ms,\msg',\mga'\right)=\emptyset\quad\mbox{ and }\quad 
Q_\phi\left(\mw\mc,\mz\ms,\msg\msg',\mga\mga'\right)=\emptyset.
\end{equation}
\end{definition}
Condition (\ref{eq:deferase}) means that the input $(\mc,\ms)$ is completely matchable, alone and together with $(\mw,\mz)$, for any fixed lists of preferences.

\begin{definition}
Let $\maB =(\maC,\maS,E,F)$ be a bipartite matching structure and $\phi$ be a matching policy. 
An admissible input $(\mc,\ms)$ is said to be a {\em strong erasing couple} for $(\maB,\phi)$ if for any words $\msg \in \mathbb C^*$ and $\mga \in \mathbb S^*$ such that $|\msg|=|\mga|=|\mc|$, 
\begin{align}
\mbox{For any suffixes } \breve{\mc},\breve{\ms},\breve{\msg}\mbox{ and }\breve{\mga} \mbox{ of }\mc,{\ms},{\msg}\mbox{ and }{\mga}\mbox{ having the same length},\quad\quad
Q_\phi\left(\breve{\mc},\breve{\ms},\breve{\msg},\breve{\mga}\right)=\emptyset,\; \label{eq:defstrongcouple1}\\
\mbox{For any }(i,j) \in E^c\mbox{ and any }(\alpha,\beta)\in \mathbb S\times \mathbb C,\,\quad\quad Q_\phi(i\mc,j\ms,\alpha\msg,\beta\mga)=\emptyset.\label{eq:defstrongcouple2} 
\end{align}
\end{definition} 
In other words, the words $(\mc,\ms)$ are completely matchable, as well as any of their suffixes of equal sizes, and this input "deletes" whatever admissible couple customer/server is present in the system. 

As is easily seen, a strong erasing couple $(\mc,\ms)$ for $(\maB,\phi)$ is an erasing couple of any single-letter buffer detail $(c,s)$, for $(c,s)\in E^c$.

\begin{lemma}
\label{lemma:erasing}
Let $\phi$ be a sub-additive matching policy, and $(\mc,\ms)$ be a strong erasing couple for $\maB$ and $\phi$. 
Then for any $(\mw,\mz) \in \CE$, we have that $\left|Q_\phi(\mw\mc,\mz\ms)\right| < \left|Q_\phi(\mw,\mz)\right|$. 
\end{lemma}

\begin{proof}
From the sub-additivity of $\phi$ and the very definition of a strong erasing couple,
\begin{equation*}
\left|Q_\phi(\mw\mc,\mz\ms)\right| \le \left|Q_\phi\left(w_1...w_{|\mw|-1},z_1...z_{|\mw|-1}\right)\right|+\left|Q_\phi\left(w_{|\mw|}\mc,z_{|\mw|}\ms\right)\right| 
                                                       =\left|Q_\phi(\mw,\mz)\right|-1.
                                                       \end{equation*} 
\end{proof}

\subsection{Structures admitting strong erasing couples}
\label{subsec:structureerasing}
Observe the following, 
\begin{proposition}
\label{pro:strongcouple}
There exists a strong erasing couple for the bipartite structure $\mathcal B=(\maC,\maS,E,F)$ and any admissible policy $\phi$, 
whenever there exist two subsets $\check{\maC} \subseteq \maC$ and $\check{\maS} \subseteq \maS$ such that 
\begin{itemize}
\item[(i)] $\maS(\check{\maC})=\maS\quad\mbox{ and }\quad \maC(\check{\maS})=\maC;$ 
\item[(ii)] Denoting by $\check{\maG}$ the sub-graph induced by $\check{\maC} \cup \check{\maS}$ in 
$(\maC \cup \maS,E)$, there exists a (non necessarily simple) alternating path $\mathscr P=i_1 \v j_1 \v i_2 \v j_2 \v ... \v i_q \v j_q$ of $\check{\maG}$ spanning the whole set $\check{\maC} \cup \check{\maS}$, and such that 
for all $\ell \in \llbracket 1,q \rrbracket,\,(i_\ell,j_\ell) \in F,$ 
\end{itemize}
and either of the three following conditions is satisfied: 
\begin{itemize}
\item[(iiia)] The sub-graph $\check{\maG}$ is bipartite complete, i.e. for all $k,\ell \in \llbracket 1,q \rrbracket$, $(i_k,j_\ell) \in E$;
\item[(iiib)] 
For all $\ell \in \llbracket 2,q \rrbracket,\,(i_\ell,j_{\ell-1})\in F$, and $\maC(j_q) \cap \check{\maC} = \{i_q\}$;
\item[(iiic)] For all $\ell \in \llbracket 2,q \rrbracket,\,(i_\ell,j_{\ell-1})\in F$, and $\maS(i_1) \cap \check{\maS} = \{j_1\}$.
\end{itemize}
\end{proposition}

\begin{proof}
All the arguments below hold true regardless of the lists of preferences of the incoming items, so we drop these parameters from all notation for short. 
Assume that (i) and (ii) hold, and fix a couple $(i,j) \in \left(\maC\times \maS\right) \setminus E.$ We examine the three alternative third conditions separately: 

\medskip
 
$\underline{\mbox{Case (iiia)}}$. Suppose that (i), (ii) and (iiia) are satisfied, and consider the couple 
\[(\mc,\ms) = \left(i_1...i_q\,,\,j_1j_2...j_q\right)\] which, from assumption (ii), clearly satisfies (\ref{eq:defstrongcouple1}).  
From (i), both $i$ and $j$ have a neighbor in $\check{\maC} \cup \check{\maS}$, and we let $j_p$ be the first letter in $\ms$ compatible with $i$, and $i_r$ be the first letter in $\mc$ that is compatible with $j$. 
We have three different cases, according to whichever comes first:  
   \begin{itemize}
   \item[(a1)] If $r=p$ (i.e. the respective neighbors of $i$ and $j$ arrive contemporarily), $i_\ell$ is matched with $j_\ell$ for any $\ell \in \llbracket 1,p-1 \rrbracket$ 
           (if the later set is non-empty). Then,  
           $i$ and $j$ are respectively matched with $j_p$ and $i_p$, and all subsequent couples are matched together on the fly, so $Q_\phi\left(i\mc,j\ms\right)=\emptyset$ (see the left display of 
   Figure \ref{Fig:strongcouple0}) and (\ref{eq:defstrongcouple2}) holds true. 
   \item[(a2)] Suppose now that $r < p$. First, for all $\ell \in \llbracket 1,r-1 \rrbracket$ (if this set is non-empty), the incoming couples $(i_\ell,j_\ell)$ are matched together upon arrival. 
   Then, $j$ is matched with $i_p$, and for all $\ell \in \llbracket r,p-1 \rrbracket$, the incoming $i_{\ell+1}$ is matched with $j_{\ell}$. Finally, 
   $j_p$ is matched with $i$ and, for all $\ell \in \llbracket p+1,q \rrbracket$ (if the later is non-empty), the incoming couples $(i_\ell,j_\ell)$ are matched together (middle display of Figure \ref{Fig:strongcouple0}), 
   and (\ref{eq:defstrongcouple2}) again holds.  
   \item[(a3)] The case $p < r$ is symmetric to (a2) (right display of Figure \ref{Fig:strongcouple0}). 
   \end{itemize} 
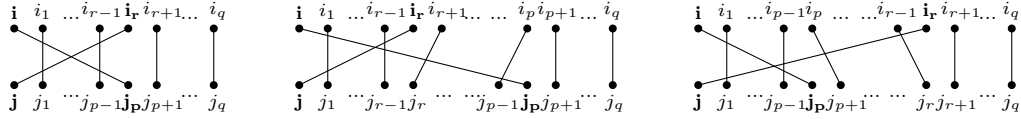
\begin{figure}[h!]
\begin{center}
\begin{tikzpicture}[scale=0.75]
\fill (-1,0) circle (2pt) node[above] {\scriptsize{$\mathbf{i}$}} ;
\fill (-0.5,0) circle (2pt) node[above] {\scriptsize{$i_1$}} ;
\fill (0,0) node[above] {\scriptsize{...}} ;
\fill (0.5,0) circle (2pt) node[above] {\,\scriptsize{$i_{r-1}$}} ;
\fill (1,0) circle (2pt) node[above] {\,\,\scriptsize{$\mathbf{i_r}$}} ;
\fill (1.5,0) circle (2pt) node[above] {\,\,\,\scriptsize{$i_{r+1}$}} ;
\fill (2,0) node[above] {\,\,\scriptsize{...}} ;
\fill (2.5,0) circle (2pt) node[above] {\,\,\scriptsize{$i_q$}} ;
\fill (-1,-1) circle (2pt) node[below] {\scriptsize{$\mathbf{j}$}} ;
\fill (-0.5,-1) circle (2pt) node[below] {\scriptsize{$j_1$}} ;
\fill (0,-1) node[below] {\scriptsize{...}};
\fill (0.5,-1) circle (2pt) node[below] {\,\scriptsize{$j_{p-1}$}} ;
\fill (1,-1) circle (2pt) node[below] {\,\,\scriptsize{$\mathbf{j_p}$}} ;
\fill (1.5,-1) circle (2pt) node[below] {\,\,\,\scriptsize{$j_{p+1}$}} ;
\fill (2,-1) node[below] {\,\,\scriptsize{...}};
\fill (2.5,-1) circle (2pt) node[below] {\,\,\scriptsize{$j_q$}} ;
\draw[-] (-1,0)-- (1,-1);
\draw[-] (-1,-1)-- (1,0);
\draw[-] (-0.5,0)-- (-0.5,-1);
\draw[-] (0.5,0)-- (0.5,-1);
\draw[-] (1.5,0)-- (1.5,-1);
\draw[-] (2.5,0)-- (2.5,-1);
\fill (4,0) circle (2pt) node[above] {\scriptsize{$\mathbf{i}$}} ;
\fill (4.5,0) circle (2pt) node[above] {\scriptsize{$i_1$}} ;
\fill (5,0) node[above] {\scriptsize{...}} ;
\fill (5.5,0) circle (2pt) node[above] {\,\scriptsize{$i_{r-1}$}} ;
\fill (6,0) circle (2pt) node[above] {\,\,\scriptsize{$\mathbf{i_r}$}} ;
\fill (6.5,0) circle (2pt) node[above] {\,\,\,\scriptsize{$i_{r+1}$}} ;
\fill (7,0) node[above] {\,\scriptsize{...}} ;
\fill (7.5,0) node[above] {\scriptsize{...}};
\fill (8,0) circle (2pt) node[above] {\scriptsize{$i_p$}} ;
\fill (8.5,0) circle (2pt) node[above] {\,\scriptsize{$i_{p+1}$}} ;
\fill (9,0) node[above] {\,\,\scriptsize{...}};
\fill (9.5,0) circle (2pt) node[above] {\scriptsize{$i_{q}$}} ;
\fill (4,-1) circle (2pt) node[below] {\scriptsize{$\mathbf{j}$}} ;
\fill (4.5,-1) circle (2pt) node[below] {\scriptsize{$j_1$}} ;
\fill (5,-1) node[below] {\scriptsize{...}} ;
\fill (5.5,-1) circle (2pt) node[below] {\,\scriptsize{$j_{r-1}$}} ;
\fill (6,-1) circle (2pt) node[below] {\,\,\scriptsize{$j_r$}} ;
\fill (6.5,-1) node[below] {\,\scriptsize{...}};
\fill (7,-1) node[below] {\,\scriptsize{...}};
\fill (7.5,-1) circle (2pt) node[below] {\scriptsize{$j_{p-1}$}} ;
\fill (8,-1) circle (2pt) node[below] {\,\,\scriptsize{$\mathbf{j_p}$}} ;
\fill (8.5,-1) circle (2pt) node[below] {\,\,\,\scriptsize{$j_{p+1}$}} ;
\fill (9,-1) node[below] {\,\,\scriptsize{...}};
\fill (9.5,-1) circle (2pt) node[below] {\,\scriptsize{$j_q$}} ;
\draw[-] (4,0)-- (8,-1);
\draw[-] (4,-1)-- (6,0);
\draw[-] (4.5,-1)-- (4.5,0);
\draw[-] (5.5,0)-- (5.5,-1);
\draw[-] (6.5,0)-- (6,-1);
\draw[-] (8,0)-- (7.5,-1);
\draw[-] (8.5,0)-- (8.5,-1);
\draw[-] (9.5,0)-- (9.5,-1);
\fill (11,0) circle (2pt) node[above] {\scriptsize{$\mathbf{i}$}} ;
\fill (11.5,0) circle (2pt) node[above] {\scriptsize{$i_1$}} ;
\fill (12,0) node[above] {\scriptsize{...}} ;
\fill (12.5,0) circle (2pt) node[above] {\,\scriptsize{$i_{p-1}$}} ;
\fill (13,0) circle (2pt) node[above] {\,\scriptsize{$i_p$}} ;
\fill (13.5,0) node[above] {\scriptsize{...}} ;
\fill (14,0) node[above] {\scriptsize{...}} ;
\fill (14.5,0) circle (2pt) node[above] {\scriptsize{$i_{r-1}$}} ;
\fill (15,0) circle (2pt) node[above] {\,\,\scriptsize{$\mathbf{i_{r}}$}} ;
\fill (15.5,0) circle (2pt) node[above] {\,\,\,\scriptsize{$i_{r+1}$}} ;
\fill (16,0) node[above] {\,\,\scriptsize{...}} ;
\fill (16.5,0) circle (2pt) node[above] {\scriptsize{$i_{q}$}} ;
\fill (11,-1) circle (2pt) node[below] {\scriptsize{$\mathbf{j}$}} ;
\fill (11.5,-1) circle (2pt) node[below] {\scriptsize{$j_1$}} ;
\fill (12,-1) node[below] {\scriptsize{...}} ;
\fill (12.5,-1) circle (2pt) node[below] {\,\scriptsize{$j_{p-1}$}} ;
\fill (13,-1) circle (2pt) node[below] {\,\,\scriptsize{$\mathbf{j_p}$}} ;
\fill (13.5,-1) circle (2pt) node[below] {\,\,\,\scriptsize{$j_{p+1}$}} ;
\fill (14,-1) node[below] {\,\scriptsize{...}};
\fill (14.5,-1) node[below] {\scriptsize{...}};
\fill (15,-1) circle (2pt) node[below] {\scriptsize{$j_r$}} ;
\fill (15.5,-1) circle (2pt) node[below] {\,\scriptsize{$j_{r+1}$}} ;
\fill (16,-1) node[below] {\,\,\scriptsize{...}};
\fill (16.5,-1) circle (2pt) node[below] {\scriptsize{$j_{q}$}} ;
\draw[-] (11,0)-- (13,-1);
\draw[-] (11,-1)-- (15,0);
\draw[-] (11.5,0)-- (11.5,-1);
\draw[-] (12.5,0)-- (12.5,-1);
\draw[-] (13,0)-- (13.5,-1);
\draw[-] (14.5,0)-- (15,-1);
\draw[-] (15.5,0)-- (15.5,-1);
\draw[-] (16.5,0)-- (16.5,-1);
\end{tikzpicture}
\caption[smallcaption]{Perfect matchings of $(i\mc,j\ms)$, case (iiia).} \label{Fig:strongcouple0}
\end{center}
\end{figure}  

\bigskip

$\underline{\mbox{Case (iiib)}}$. Suppose now that (i), (ii) and (iiib) hold, and let $(\mc,\ms)$ be the following couple of length $2q-1$, 
\[(\mc,\ms) = \left(i_1...i_qi_2i_3...i_q\,,\,j_1j_2...j_qj_1j_2...j_{q-1}\right),\]
which is admissible from the first assertion of assumption (iiib), and clearly satisfies (\ref{eq:defstrongcouple1}) by the very definition of $\maP$.  
Again from (i), both $i$ and $j$ have a neighbor in $\maP$, 
and we let again $j_p$ be the first compatible server with $i$ in $\mc$ and $i_r$ be the first compatible customer with $j$ in $\ms$. The three sub-cases are as above: 
           \begin{itemize}
           \item[(b1)] $p=r$. This case is analog to (a1);
           \item[(b2)] $p>r$, that is, the first match of $j$ arrives before that of $i$. This case is analog to (a2);
           \item[(b3)] $p<r$. First, $i_\ell$ is matched with $j_\ell$ for any $\ell \in \llbracket 1,p-1 \rrbracket$ if the later set is non-empty. 
           Then $i$ is matched with $j_p$, and we apply the following procedure: 
                \begin{itemize}
                \item we set $p_0:= p;$ 
                \item as long as the set \[\maA_k:=\{\ell \in \llbracket p_{k}+1,r-1 \rrbracket \,:\,i_{p_{k}} \v j_\ell\}\]
                is non-empty, we set $p_{k+1}=\min \maA_k.$ 
                \item we let $m$ be the smallest index $k$ such that $\maA_k = \emptyset$.    
                \end{itemize}
          Observe that $m \ge 1$. Then, for any $k \in \llbracket 0,m-1 \rrbracket$, $i_{p_k}$ is matched with $j_{p_{k+1}}$, and for any $\ell \in \llbracket p_k+1,p_{k+1}-1 \rrbracket$ 
         (if the latter is non-empty), $i_\ell$ is matched with $j_\ell$. Finally, for any $\ell \in \llbracket p_m+1,r-1 \rrbracket$ (if the latter set is non-empty), $i_\ell$ is matched with $j_\ell$, and $j$ is matched with 
        $i_r$. Thus, after $r$ arrivals and whatever the matching policy is, we are in the following alternative: either $i_{p_m} \v j_r$, in which case these two items are matched together and the system is empty, or 
        $i_{p_m} \pv j_r$, and then only the couple $(i_{p_m},j_r)$ remains to be matched. In the first case, as it is clear that all subsequent entering couples of $(\mc,\ms)$ are matched together upon arrival, so 
        $Q_\phi(i\mc,j\ms)=\emptyset.$ Only the second case remains to be treated, in other words we are rendered to prove that 
        \begin{equation}
        \label{eq:strongcouple3}
        Q_\phi\left(i_{p_m}i_{r+1}...i_qi_2i_3...i_q,j_rj_{r+1}...j_qj_1...j_{q-1}\right)=\emptyset. 
        \end{equation}
        Again, we have two sub-cases: 
              \begin{itemize}
                 \item[$\star$] There exists an index $k \in \llbracket r+1,q\rrbracket$ such $i_{p_m} \v j_k$ (take the smallest such $k$). Then, 
                 whatever $\phi$ is, from the buffer-first assumption, for all $\ell \in \llbracket r,k-1 \rrbracket$, $j_\ell$ is matched with $i_{\ell +1}$, and then $i_{r_m}$ is matched 
                 with $j_k$. Thus the matching is complete after $k$ arrivals, and $\phi$ matches all incoming couples on the fly after that, so (\ref{eq:strongcouple3}) holds; see the top display of Figure \ref{Fig:strongcouple}. 
                 \item[$\star\star$] $r=q$, or $i_{p_m}$ does not find a match in the set $\{j_{r+1},...,j_q\}$. Then, for all $\ell \in k \in \llbracket r,q-1 \rrbracket$ (if the latter is non-empty), 
                 $\phi$ matches $i_\ell$ with $j_\ell$, so after $q$ arrivals only the couple $(i_{r_m},j_q)$ remains to be matched. Now, observe that the set 
                 \[\maB:=\{\ell \in \llbracket 1,p_m \rrbracket \,:\,i_{p_{m}} \v j_\ell\}\] 
                 is non-empty, since, by the definition of $\maP$, $i_{p_m} \v j_1$ if $p_m=1$, and otherwise $i_{p_m} \v j_{p_m-1}$ and $i_{p_m} \v j_{p_m}$. 
                 We let $n:=\min\maB$. Then, from the second assertion of (iiib), $j_q$ necessarily waits for the next $i_q$ to find a compatible item. Therefore, before that, 
                 $\phi$ necessarily matches, for any $\ell \in \llbracket 1,n-1 \rrbracket$ (if the latter is non-empty), the second incoming $j_\ell$ with the second incoming $i_{\ell+1}$. 
                 After that, $i_{r_m}$ is matched with $j_n$, and then for all $\ell \in \llbracket n,q-1 \rrbracket$ (if the latter is non-empty), $\phi$ matched the second $i_\ell$ with the second $j_\ell$. 
                 Finally, the still unmatched $j_q$ is matched with the second entering $i_q$ (see the bottom display of Figure \ref{Fig:strongcouple}). 
                Again, (\ref{eq:strongcouple3}) holds true, which concludes the proof of (\ref{eq:defstrongcouple2}) in that case.  
                \end{itemize} 
        \end{itemize}     
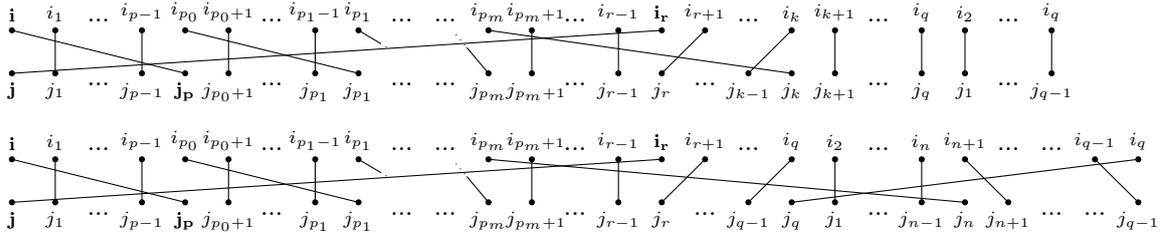
\begin{figure}[h!]
\begin{center}
\begin{tikzpicture}[scale=0.57]
\fill (4,0) circle (2pt) node[above] {\scriptsize{$\mathbf{i}$}} ;
\fill (5,0) circle (2pt) node[above] {\scriptsize{$i_1$}} ;
\fill (6,0) node[above] {...} ;
\fill (7,0) circle (2pt) node[above] {\scriptsize{$i_{p-1}$}} ;
\fill (8,0) circle (2pt) node[above] {\scriptsize{$i_{p_0}$}} ;
\fill (9,0) circle (2pt) node[above] {\scriptsize{$i_{p_0+1}$}} ;
\fill (10,0) node[above] {...} ;
\fill (11,0) circle (2pt) node[above] {\scriptsize{$i_{p_1-1}$}};
\fill (12,0) circle (2pt) node[above] {\scriptsize{$i_{p_1}$}};
\fill (13,0) node[above] {...} ;
\fill (14,0) node[above] {...} ;
\fill (15,0) circle (2pt) node[above] {\scriptsize{$i_{p_{m}}$}} ;
\fill (16,0) circle (2pt) node[above] {\,\,\scriptsize{$i_{p_{m}+1}$}} ;
\fill (17,0) node[above] {...} ;
\fill (18,0) circle (2pt) node[above] {\scriptsize{$i_{r-1}$}};
\fill (19,0) circle (2pt) node[above] {\scriptsize{$\mathbf{i_r}$}};
\fill (20,0) circle (2pt) node[above] {\scriptsize{$i_{r+1}$}};
\fill (21,0) node[above] {...};
\fill (22,0) circle (2pt) node[above] {\scriptsize{$i_{k}$}};
\fill (23,0) circle (2pt) node[above] {\scriptsize{$i_{k+1}$}};
\fill (24,0) node[above] {...} ;
\fill (25,0) circle (2pt) node[above] {\scriptsize{$i_{q}$}}; 
\fill (26,0) circle (2pt) node[above] {\scriptsize{$i_{2}$}};
\fill (27,0) node[above] {...} ;
\fill (28,0) circle (2pt) node[above] {\scriptsize{$i_{q}$}};
\fill (4,-1) circle (2pt) node[below] {\scriptsize{$\mathbf{j}$}} ;
\fill (5,-1) circle (2pt) node[below] {\scriptsize{$j_1$}} ;
\fill (6,-1) node[below] {...} ;
\fill (7,-1) circle (2pt) node[below] {\scriptsize{$j_{p-1}$}} ;
\fill (8,-1) circle (2pt) node[below] {\scriptsize{$\mathbf{j_{p}}$}} ;
\fill (9,-1) circle (2pt) node[below] {\scriptsize{$j_{p_0+1}$}} ;
\fill (10,-1) node[below] {...} ;
\fill (11,-1) circle (2pt) node[below] {\scriptsize{$j_{p_1}$}};
\fill (12,-1) circle (2pt) node[below] {\scriptsize{$j_{p_1}$}};
\fill (13,-1) node[below] {...} ;
\fill (14,-1) node[below] {...} ;
\fill (15,-1) circle (2pt) node[below] {\scriptsize{$j_{p_{m}}$}} ;
\fill (16,-1) circle (2pt) node[below] {\,\,\scriptsize{$j_{p_{m}+1}$}} ;
\fill (17,-1) node[below] {...} ;
\fill (18,-1) circle (2pt) node[below] {\scriptsize{$j_{r-1}$}};
\fill (19,-1) circle (2pt) node[below] {\scriptsize{$j_{r}$}};
\fill (20,-1) node[below] {...} ;
\fill (21,-1) circle (2pt) node[below] {\scriptsize{$j_{k-1}$}};
\fill (22,-1) circle (2pt) node[below] {\scriptsize{$j_{k}$}};
\fill (23,-1) circle (2pt) node[below] {\scriptsize{$j_{k+1}$}};
\fill (24,-1) node[below] {...} ;
\fill (25,-1) circle (2pt) node[below] {\scriptsize{$j_{q}$}}; 
\fill (26,-1) circle (2pt) node[below] {\scriptsize{$j_{1}$}};
\fill (27,-1) node[below] {...} ;
\fill (28,-1) circle (2pt) node[below] {\scriptsize{$j_{q-1}$}}; 
\draw[-] (4,0)-- (8,-1);
\draw[-] (4,-1)-- (19,0);
\draw[-] (5,-1)-- (5,0);
\draw[-] (7,-1)-- (7,0);
\draw[-] (12,-1)-- (8,0);
\draw[-] (9,0)-- (9,-1);
\draw[-] (11,0)-- (11,-1);
\draw[-] (12,0)-- (12.5,-0.3);
\draw[-,dotted] (12.5,-0.3)-- (12.75,-0.45);
\draw[-,dotted] (14.25,-0.1)-- (14.5,-0.4);
\draw[-] (14.5,-0.4)-- (15,-1);
\draw[-] (16,0)-- (16,-1);
\draw[-] (18,0)-- (18,-1);
\draw[-] (20,0)-- (19,-1);
\draw[-] (22,0)-- (21,-1);
\draw[-] (15,0)-- (22,-1);
\draw[-] (23,0)-- (23,-1);
\draw[-] (25,0)-- (25,-1);
\draw[-] (26,0)-- (26,-1);
\draw[-] (28,0)-- (28,-1);
\fill (4,-3) circle (2pt) node[above] {\scriptsize{$\mathbf{i}$}} ;
\fill (5,-3) circle (2pt) node[above] {\scriptsize{$i_1$}} ;
\fill (6,-3) node[above] {...} ;
\fill (7,-3) circle (2pt) node[above] {\scriptsize{$i_{p-1}$}} ;
\fill (8,-3) circle (2pt) node[above] {\scriptsize{$i_{p_0}$}} ;
\fill (9,-3) circle (2pt) node[above] {\scriptsize{$i_{p_0+1}$}} ;
\fill (10,-3) node[above] {...} ;
\fill (11,-3) circle (2pt) node[above] {\scriptsize{$i_{p_1-1}$}};
\fill (12,-3) circle (2pt) node[above] {\scriptsize{$i_{p_1}$}};
\fill (13,-3) node[above] {...} ;
\fill (14,-3) node[above] {...} ;
\fill (15,-3) circle (2pt) node[above] {\scriptsize{$i_{p_{m}}$}} ;
\fill (16,-3) circle (2pt) node[above] {\,\,\scriptsize{$i_{p_{m}+1}$}} ;
\fill (17,-3) node[above] {...} ;
\fill (18,-3) circle (2pt) node[above] {\scriptsize{$i_{r-1}$}};
\fill (19,-3) circle (2pt) node[above] {\scriptsize{$\mathbf{i_r}$}};
\fill (20,-3) circle (2pt) node[above] {\scriptsize{$i_{r+1}$}};
\fill (21,-3) node[above] {...};
\fill (22,-3) circle (2pt) node[above] {\scriptsize{$i_{q}$}};
\fill (23,-3) circle (2pt) node[above] {\scriptsize{$i_{2}$}};
\fill (24,-3) node[above] {...} ;
\fill (25,-3) circle (2pt) node[above] {\scriptsize{$i_{n}$}}; 
\fill (26,-3) circle (2pt) node[above] {\scriptsize{$i_{n+1}$}};
\fill (27,-3) node[above] {...} ;
\fill (28,-3) node[above] {...} ;
\fill (29,-3) circle (2pt) node[above] {\scriptsize{$i_{q-1}$}};
\fill (30,-3) circle (2pt) node[above] {\scriptsize{$i_{q}$}};
\fill (4,-4) circle (2pt) node[below] {\scriptsize{$\mathbf{j}$}} ;
\fill (5,-4) circle (2pt) node[below] {\scriptsize{$j_1$}} ;
\fill (6,-4) node[below] {...} ;
\fill (7,-4) circle (2pt) node[below] {\scriptsize{$j_{p-1}$}} ;
\fill (8,-4) circle (2pt) node[below] {\scriptsize{$\mathbf{j_{p}}$}} ;
\fill (9,-4) circle (2pt) node[below] {\scriptsize{$j_{p_0+1}$}} ;
\fill (10,-4) node[below] {...} ;
\fill (11,-4) circle (2pt) node[below] {\scriptsize{$j_{p_1}$}};
\fill (12,-4) circle (2pt) node[below] {\scriptsize{$j_{p_1}$}};
\fill (13,-4) node[below] {...} ;
\fill (14,-4) node[below] {...} ;
\fill (15,-4) circle (2pt) node[below] {\scriptsize{$j_{p_{m}}$}} ;
\fill (16,-4) circle (2pt) node[below] {\,\,\scriptsize{$j_{p_{m}+1}$}} ;
\fill (17,-4) node[below] {...} ;
\fill (18,-4) circle (2pt) node[below] {\scriptsize{$j_{r-1}$}};
\fill (19,-4) circle (2pt) node[below] {\scriptsize{$j_{r}$}};
\fill (20,-4) node[below] {...} ;
\fill (21,-4) circle (2pt) node[below] {\scriptsize{$j_{q-1}$}};
\fill (22,-4) circle (2pt) node[below] {\scriptsize{$j_{q}$}};
\fill (23,-4) circle (2pt) node[below] {\scriptsize{$j_{1}$}};
\fill (24,-4) node[below] {...} ;
\fill (25,-4) circle (2pt) node[below] {\scriptsize{$j_{n-1}$}}; 
\fill (26,-4) circle (2pt) node[below] {\scriptsize{$j_{n}$}};
\fill (27,-4) circle (2pt) node[below] {\scriptsize{$j_{n+1}$}};
\fill (28,-4) node[below] {...} ;
\fill (29,-4) node[below] {...} ;
\fill (30,-4) circle (2pt) node[below] {\scriptsize{$j_{q-1}$}}; 
\draw[-] (4,-3)-- (8,-4);
\draw[-] (4,-4)-- (19,-3);
\draw[-] (5,-4)-- (5,-3);
\draw[-] (7,-4)-- (7,-3);
\draw[-] (12,-4)-- (8,-3);
\draw[-] (9,-3)-- (9,-4);
\draw[-] (11,-3)-- (11,-4);
\draw[-] (12,-3)-- (12.5,-3.3);
\draw[-,dotted] (12.5,-3.3)-- (12.75,-3.45);
\draw[-,dotted] (14.25,-3.1)-- (14.5,-3.4);
\draw[-] (14.5,-3.4)-- (15,-4);
\draw[-] (16,-3)-- (16,-4);
\draw[-] (18,-3)-- (18,-4);
\draw[-] (20,-3)-- (19,-4);
\draw[-] (22,-3)-- (21,-4);
\draw[-] (15,-3)-- (26,-4);
\draw[-] (23,-3)-- (23,-4);
\draw[-] (25,-3)-- (25,-4);
\draw[-] (26,-3)-- (27,-4);
\draw[-] (29,-3)-- (30,-4);
\draw[-] (30,-3)-- (22,-4);
\end{tikzpicture}
\caption[smallcaption]{Perfect matchings of $(i\mc,j\ms)$ in case (b3): sub-case $\star$ (above) and sub-case $\star \star$ (below).} \label{Fig:strongcouple}
\end{center}
\end{figure} 
\bigskip 
$\underline{\mbox{Case (iiic)}}$. By symmetry between customers and servers and reading words in reverse sense, we can apply the exact argument of (iiib) for the input  
\[(\mc,\ms) = \left(i_q...i_1i_{q}...i_{2}\,,\,j_q...j_1j_{q-1}...j_{1}\right).\]
\end{proof}

With Proposition \ref{pro:strongcouple} in hand, we can now make more precise the classes of bipartite structures admitting strong erasing couples. 

\paragraph{Bi-separable graphs.} Strong erasing couples are easily obtained for bipartite structures having a 
bi-separable matching graph: 

\begin{proposition}
\label{pro:strongerasesep}
Suppose that the matching graph $\maG=(\maC\cup\maS,E)$ is bi-separable of order at least 3. 
If there exist three maximal independent sets $I_1$, $I_2$ and $I_3$ and three couples 
$(k_1,\ell_1)\in I_1,\,(k_2,\ell_2)\in I_2$ and $(k_3,\ell_3)\in I_3$ such that $F$ contains 
$(k_1,\ell_2)$, $(k_2,\ell_3)$ and $(k_3,\ell_1)$, then $(\mathcal B,\phi)$ admits at least one strong erasing couple for any admissible policy $\phi$.
\end{proposition}

\begin{proof}
Fix an admissible matching policy. The proof does not depend in any list of preferences, and we skip again this parameter from all notation for short. 
A strong erasing word is given by $(\mc,\ms)=(k_1k_2k_3,\ell_2\ell_3\ell_1)$. To see this, it is first immediate to observe that 
$Q_\phi(\mc,\ms)=\emptyset$ by the definition of a bi-separable graph. Any single-letter buffer detail is of the form $(c,s)$, where $c\in A$, $s \in B$ and $A\cup B$ is a maximal independent set of $\maG$. Thus we have four alternatives: 
\begin{itemize}
\item If $I=I_1$, then $M_\phi(c\mc,s\ms)$ contains the matches $(c,\ell_2)$, $(k_1,\ell_3)$, $(k_2,s)$ and $(k_3,\ell_1)$;
\item If $I=I_2$, then $M_\phi(c\mc,s\ms)$ contains $(c,\ell_3)$, $(k_1,s)$, $(k_2,\ell_1)$ and $(k_3,\ell_2)$;
\item If $I=I_3$, $M_\phi(c\mc,s\ms)$ contains $(c,\ell_2)$, $(k_1,s)$, $(k_2,\ell_3)$ and $(k_3,\ell_1)$;
\item If $\maG$ is of order at least 4 and $I\not\in\{I_1,I_2,I_3\}$, then $M_\phi(c\mc,s\ms)$ contains $(c,\ell_2)$, $(k_1,s)$, $(k_2,\ell_3)$ and $(k_3,\ell_1)$.  
\end{itemize}
In all cases we have $Q_\phi(c\mc,s\ms)=\emptyset,$ which concludes the proof.   
\end{proof}

\paragraph{GM model.} Observe that condition (ii) of Proposition \ref{pro:strongcouple} cannot hold for a GM model (i.e. $F=\left\{(c,\tilde c), c\in\maC\right\}$), as we would have 
$c \v c' = \tilde c$ for some $c \in\maC$, an absurdity. Proposition 7 in \cite{MBM17} provides alternative specific conditions for the existence of 
strong erasing words (analog to strong erasing couples in that case) for BM models. 

\paragraph{BM Model.} Consider now the case where $(\maC \cup \maS, F)$ is complete, i.e. a BM model. It is then immediate to reformulate the assumptions 
of Proposition \ref{pro:strongcouple}, to obtain the following: 

\begin{proposition}
\label{pro:strongcoupleBM}
Suppose that $\maB$ defines a BM model. If, for some subsets $\check{\maC}$ and $\check{\maS}$ such that $\maS(\check{\maC})=S$ and $\maC(\check{\maS})=C,$ there exists an alternating path $\mathscr P=i_1 \v j_1 \v i_2 \v j_2 \v ... \v i_q \v j_q$ spanning $\check{\maC} \cup \check{\maS}$, 
and such that either one of the following hold: 
\begin{itemize}
\item[(a)] the sub-graph $\check{\maG}$ induced by $\check{\maC} \cup \check{\maS}$ in 
$(\maC \cup \maS,E)$ is bipartite complete;
\item[(b)] $\maC(j_q) \cap \check{\maC} = \{i_q\}$;
\item[(c)] $\maS(i_1) \cap \check{\maS} = \{j_1\}$,
\end{itemize}
Then $(\mathcal B,\phi)$ admits at least one strong erasing couple for any admissible policy $\phi$. 
\end{proposition}

Clearly, condition (b) (respectively, (c)) above holds for $\check{\maC} \equiv\maC$ and $\check{\maS} \equiv\maS$ (as $(\maC\cup\maS,E)$ is connected), whenever $(\maC \cup \maS,E)$ contains a customer (resp. server) node of degree 1, by setting 
$i_1$ (resp., $j_q$) as the latter node.

\subsection{Structures admitting erasing couples} 

The following result provides sufficient conditions for the existence of erasing couples for any admissible buffer detail, 
\begin{proposition}
\label{prop:existerasing}
Let $\phi$ be a sub-additive policy. Then, any admissible buffer detail $(\mw,\mz)$ admits an erasing couple for $(\maB,\phi)$ in the following cases: 
\begin{enumerate}
\item $\maB$ satisfies Proposition \ref{pro:strongcouple} or Proposition \ref{pro:strongerasesep};
\item $\maB$ defines a BM model (i.e. $F=\maC\times\maS$);
\item $\maB$ defines a GM model and $(\mw,\mz)$ is of even length.  
\end{enumerate}
\end{proposition}  

\begin{proof} 
The argument in the proof are clearly independent of the lists of preferences, and we drop again this  parameter of all notations for short. 
\begin{enumerate}
\item Suppose that Proposition \ref{pro:strongcouple} is satisfied, and thereby a strong erasing couple $(\mc,\ms)$ exists for $(\maB,\phi)$. 
Therefore, as was observed above, $(\mc,\ms)$ is an erasing couple of any single-letter couple $(i,j)\in\maC\times\maS \setminus E$.  

Let us now consider an admissible buffer detail $(\mw,\mz)=\left(w_1...w_r\;,\;z_1...z_r\right)$ for $r\ge 1$. 
First, as we just proved, there exists an erasing couple, say $(\mc^1,\ms^1)$, for the single-letter word $(w_r,z_r)$. 
Thus, the sub-additivity of $\phi$ entails that 
\begin{align*}
\left|C_\phi\left(\mw\mc^1,\mz\ms^1\right)\right| &\le \left|C_\phi\left(w_1w_2...w_{r-1}\,,\,z_1z_2...z_{r-1}\right)\right|+\left|C_\phi\left(w_r\mc^1\,,\,z_r\ms^1\right)\right| 
= \left|C_\phi(\mw,\mz)\right|-1;\\
\left|S_\phi\left(\mw\mc^1,\mz\ms^1\right)\right| &\le \left|S_\phi\left(w_1w_2...w_{r-1}\,,\,z_1z_2...z_{r-1}\right)\right|+\left|S_\phi\left(w_r\mc^1\,,\,z_r\ms^1\right)\right| 
= \left|S_\phi(\mw,\mz)\right|-1.
\end{align*}
We can in turn, apply the same argument to the admissible buffer detail $Q_\phi\left(\mw\mc^1,\mz\ms^1\right)$: there exists an erasing couple $(\mc^2,\ms^2)$ for the single-letter 
couple gathering the last letter of $C_\phi\left(\mw\mc^1,\mz\ms^1\right)$ and the las letter of $S_\phi\left(\mw\mc^1,\mz\ms^1\right)$. As above, the sub-additivity of $\phi$ entails that 
\begin{align*}
\left|C_\phi\left(\mw\mc^1\mc^2,\mz\ms^1\ms^2\right)\right| &= \left|C_\phi\left(C_\phi\left(\mw\mc^1,\mz\ms^1\right)\mc^2\,,\,S_\phi\left(\mw\mc^1,\mz\ms^1\right)\ms^2\right)\right| 
\le \left|C_\phi\left(\mw\mc^1,\mz\ms^1\right)\right|-1;\\
\left|S_\phi\left(\mw\mc^1\mc^1,\mz\ms^1\ms^2\right)\right| &= \left|S_\phi\left(C_\phi\left(\mw\mc^1,\mz\ms^1\right)\mc^2\,,\,S_\phi\left(\mw\mc^1,\mz\ms^1\right)\ms^2\right)\right| 
\le \left|S_\phi\left(\mw\mc^1,\mz\ms^1\right)\right|-1.
\end{align*}
By an immediate induction, we obtain that for some $p \le r$, there exist $p$ couples $(\mc^1,\ms^1)$, ... , $(\mc^p,\ms^p)$ such that 
\begin{equation}
\label{eq:OK}
Q_\phi\left(\mw\mc^1...\mc^p\,,\,\mz\ms^1...\ms^p\right) = \left(C_\phi\left(\mw\mc^1...\mc^p\,,\,\mz\ms^1...\ms^p\right)\,,\,S_\phi\left(\mw\mc^1...\mc^p\,,\,\mz\ms^1...\ms^p\right)\right) 
=\emptyset.
\end{equation}
On the other hand, by the very definition of an erasing couple we have that $Q_\phi(\mc^1,\ms^1)=...=Q_\phi(\mc^p,\ms^p)=\emptyset.$ 
Thus $Q_\phi(\mc^1...\mc^p,\ms^1...\ms^p)=\emptyset$, which shows, together with (\ref{eq:OK}), that $(\mc^1...\mc^p,\ms^1...\ms^p)$ is an erasing couple for $(\mw,\mz)$.
\item As $\maG=(\maC\cup\maS,E)$ is connected, there exists an alternating path 
\[i \v j_1 \v i_1 \v j_2 \v \,...\, i_{p-1} \v j_p \v i_p \v j\]
connecting $i$ to $j$. Then, as $F=\maC\times\maS$, the input  
\[(\mc,\ms) = \left(i_1i_2 \,...\,i_k\,,\,j_1j_2 \,...\,j_k\right)\]
is admissible, and is clearly such that $Q_\phi(i\mc,j\ms)=\emptyset$ for any $\phi$ and any lists of preference, see Figure \ref{Fig:weakcouple}. 
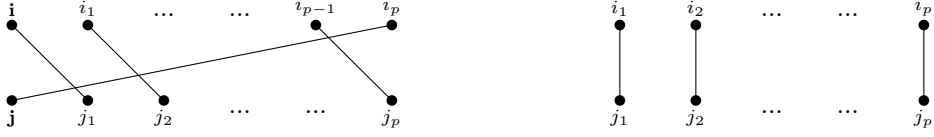
\begin{figure}[h!]
\begin{center}
\begin{tikzpicture}[scale=1]
\fill (0,0) circle (2pt) node[above] {\scriptsize{$\mathbf{i}$}} ;
\fill (1,0) circle (2pt) node[above] {\scriptsize{$i_1$}} ;
\fill (2,0) node[above] {...} ;
\fill (3,0) node[above] {...} ;
\fill (4,0) circle (2pt) node[above] {\scriptsize{$i_{p-1}$}} ;
\fill (5,0) circle (2pt) node[above] {\scriptsize{$i_{p}$}} ;
\fill (0,-1) circle (2pt) node[below] {\scriptsize{$\mathbf{j}$}} ;
\fill (1,-1) circle (2pt) node[below] {\scriptsize{$j_1$}} ;
\fill (2,-1) circle (2pt) node[below] {\scriptsize{$j_2$}} ;
\fill (3,-1) node[below] {...} ;
\fill (4,-1) node[below] {...} ;
\fill (5,-1) circle (2pt) node[below] {\scriptsize{$j_{p}$}} ;
\draw[-] (0,0) -- (1,-1);
\draw[-] (0,-1)-- (5,0);
\draw[-] (1,0)-- (2,-1);
\draw[-] (4,0)-- (5,-1);
 %
 %
\fill (8,0) circle (2pt) node[above] {\scriptsize{$i_1$}} ;
\fill (9,0) circle (2pt) node[above] {\scriptsize{$i_2$}} ;
\fill (10,0) node[above] {...} ;
\fill (11,0) node[above] {...} ;
\fill (12,0) circle (2pt) node[above] {\scriptsize{$i_{p}$}} ;
  %
\fill (8,-1) circle (2pt) node[below] {\scriptsize{$j_1$}} ;
\fill (9,-1) circle (2pt) node[below] {\scriptsize{$j_2$}} ;
\fill (10,-1) node[below] {...} ;
\fill (11,-1) node[below] {...} ;
\fill (12,-1) circle (2pt) node[below] {\scriptsize{$j_{p}$}} ;
\draw[-] (8,-1)-- (8,0);
\draw[-] (9,0)-- (9,-1);
\draw[-] (12,0)-- (12,-1);
\end{tikzpicture}
\caption[smallcaption]{Perfect matchings of $(i\mc,j\ms)$ (left) and $(\mc,\ms)$ (right).} \label{Fig:weakcouple}
\end{center}
\end{figure}
The existence of an erasing word for any arbitrary admissible buffer detail $(\mw,\mz)$ follows as in case 1. 
\item In the case where $\maB$ defines a GM model, the existence of erasing couples for all buffer details of even size is 
proven in Proposition 6 of \cite{MBM17}, transposing to the present context, the concept of {\em erasing word} therein. 
\end{enumerate}
\end{proof}

\section{Loynes construction}
\label{sec:loynes}

\subsection{Probabilistic assumptions} 
\label{subsec:NcondScond}
Fix a bipartite matching structure $\maB=(\maC,\maS,E,F)$ and a matching policy $\phi$. We suppose that the classes of the entering couples and their lists of preferences are {random}: on the reference probability space $(\Omega,\mathcal F, \mathbb P)$, we define the $F\times \mathbb S\times \mathbb C$-valued bi-infinite input sequence $\left(C_n,S_n,\Sigma_n,\Gamma_n\right)_{n\in \Z}$, and make the following general assumption: 
\begin{enumerate}
\item[\textbf{(H1)}] The sequence $\left(C_n,S_n,\Sigma_n,\Gamma_n\right)_{n\in \Z}$ is stationary, drawn at all $n$ from a distribution on $F\times \mathbb S\times \mathbb C$ whose $F$-marginal is denoted $\mu$ and $\mathbb S\times \mathbb C$-marginal is denoted $\nu^\phi$, and ergodic. 
\end{enumerate}
In the above definition, we assume that the probability measure $\mu$ has full support $F$. We denote by $\mu_{\maC}$ (respectively, $\mu_{\maS}$) its $\maC$-marginal (resp., $\maS$-marginal) distribution. 
On another hand, we emphasize the dependence of the distribution of lists of preferences on $\mathbb S \times \mathbb C$ on the matching policy. For instance:  
\begin{itemize}
\item if $\phi$ is a strict priority, i.e. the order of priority of any customer/server among the class of compatible servers/customers is fixed beforehand. 
In other word we fix a list of customer preference $\alpha$ in $\mathbb S$ and a list of server preference $\beta$ in $\mathbb C$, and we set $\nu^\phi=\delta_\alpha\otimes \delta_\beta$. 
\item if $\nu^\phi$ is the uniform distribution on $\mathbb S\times\mathbb C$, we say that the matching policy is {\em uniform}; 
\item for $\phi=\textsc{ml}$ or {\sc ms}, we also fix $\nu^\phi$ as the uniform distribution on $\mathbb S\times\mathbb C$ - i.e. ties are broken uniformly at random; 
\item the drawn lists of preference are irrelevant for policies such as {\sc fcfs} or {\sc lcfs}, and in such case we drop this parameter from all notation.   
\end{itemize} 
We will also consider the following statistical assumption which, from Birkhoff's ergodic Theorem, is strictly stronger than (H1), 
\begin{enumerate}
\item[\textbf{(IID)}] The sequence $\left(C_n,S_n,\Sigma_n,\Gamma_n\right)_{n\in \Z}$ is i.i.d., drawn at all $n$ from a distribution on $F\times \mathbb S\times \mathbb C$ whose $F$-marginal is denoted $\mu$ and $\mathbb S\times \mathbb C$-marginal is denoted $\nu^\phi$. 
\end{enumerate} 
Assumption (IID), which allows a representation of the system by a discrete-time Markov chain, is typically made in all references on matching models except \cite{MBM17}. 
Observe that the BM model, as investigated in \cite{ABMW17}, makes, in addition to (IID), the assumption that the classes of the incoming server and customer are independent, i.e. 
$\mu=\mu_{\maC} \otimes \mu_{\maS}$. We will not restrict to this case here. 

\medskip 

Consider the two following conditions on $\maG$ and $\mu$,  
\begin{equation}
\label{eq:Ncond}
\text{For any set }A\subsetneq\maC\mbox{ and }B\subsetneq\maS,\quad\quad \mu_\maC(A) < \mu_{\maS}\left(\maS(A)\right)\quad\quad\mbox{ and }\quad\quad\mu_\maS(B) < \mu_{\maC}\left(\maC(B)\right),
\end{equation} 
and recalling (\ref{eq:defCSnot}),                     
\begin{equation}
\label{eq:Scond}
\text{For any independent set }I=A \cup B \ne \emptyset,\quad\mu_\maC\left(\maC(B)\right)+\mu_\maS(\maS(A)) >1-\mu\left(E \cap \left(\maC_\circ\left(I\right)\times \maS_\circ(I)\right)\right).
\end{equation}
The condition (\ref{eq:Ncond}) was introduced in \cite{calkapwei09}, and shown to guarantee complete resource pooling in a BM system under the policy {\sc fcfs}, and for $\mu=\mu_{\maC}\otimes \mu_{\maS}.$ 
It was shown in Lemma 3.2 in \cite{BGM13} to be necessary for the positive recurrence of any EBM system under the (IID) assumption. 
Condition (\ref{eq:Scond}) was also introduced in \cite{BGM13}, and shown to be sufficient for the positive recurrence of an EBM system under (IID) (Proposition 5.2 of \cite{BGM13}). 

\medskip

Let us now introduce the following condition for a bi-separable matching graph $\maG=(\maC \cup \maS, E)$ of order $p$, denoting by  
$I_1=(A_1 \cup B_1),...,I_p=(A_p \cup B_p)$ the independent sets of the corresponding partition of $\maC \cup \maS$, 
\begin{equation}
\label{eq:scondmonotone0}
\mbox{For all }i\in\llbracket 1,p \rrbracket,\, \mu\left(A_i\times B_i\right) < {1 \over 2}.
\end{equation}
Observe that (\ref{eq:scondmonotone}) is non-empty if and only $\maG$ is of order strictly greater than 2. 
We have the following result, 
\begin{proposition}
\label{pro:equivalenceBGM}
For a bi-separable graph $\maG=(\maC \cup \maS, E)$ of order $p$, both conditions (\ref{eq:Ncond}) and (\ref{eq:Scond}) are equivalent to (\ref{eq:scondmonotone0}).
\end{proposition}

\begin{proof}
Denote again $I_1=A_i \cup B_i,...,I_p=A_p\cup B_p$ the independent sets of the partition of $\maC \cup \maS$. It is an immediate consequence of Lemma \ref{lemma:separable} that 
(\ref{eq:scondmonotone0}) is equivalent to  
\begin{equation}
\label{eq:scondmonotone}
\mbox{For all }i\in\llbracket 1,p \rrbracket,\, \mu\left(A_i\times B_i\right) < \mu\left(\maC(B_i)\times \maS(A_i)\right).
\end{equation}
Thus, as (\ref{eq:Scond}) entails (\ref{eq:Ncond}), it suffices to check that 
(\ref{eq:Ncond}) entails (\ref{eq:scondmonotone}) and (\ref{eq:scondmonotone}) entails (\ref{eq:Scond}).

\medskip

\noindent $\underline{(\ref{eq:Ncond}) \,\Rightarrow\, (\ref{eq:scondmonotone})}$:  
Suppose that $A_i \ne \emptyset$ and $B_i \ne \emptyset$ (otherwise the result is trivial). We deduce from (\ref{eq:Ncond}) that
\begin{align*}
\mu\Bigl(A_i\times B_i\Bigl) =\mu_\maC\left(A_i\right)-\mu\Bigl(A_i\times \maS\left(A_i\right)\Bigl)
                                &<\mu_\maS\left(\maS\left(A_i\right)\right)-\mu\Bigl(A_i\times \maS\left(A_i\right)\Bigl)\\
                                &=\mu\Bigl(\maC\left(B_i\right)\times \maS\left(A_i\right) \Bigl).
\end{align*}

\noindent $\underline{(\ref{eq:scondmonotone})\,\Rightarrow\,(\ref{eq:Scond})}$: let $I= A\cup B$ be a non-empty independent set of $\maG$.
\begin{itemize}
\item[(i)] If $I=I_i$ for some $i\in\llbracket 1,p \rrbracket$,
\begin{align}
\mu\Bigl(\maC\left(B_i\right)\times \maS\left(A_i\right) \Bigl) > \mu\left(A_i\times B_i\right)
&\Longleftrightarrow \mu_\maC\left(\maC(B_i)\right)-\mu\Bigl(\maC(B_i)\times B_i\Bigl)> \mu\left(A_i\times B_i\right)\nonumber\\
&\Longleftrightarrow \mu_\maC\left(\maC(B_i)\right)> \mu_\maS\left(B_i\right)\nonumber\\
&\Longleftrightarrow \mu_\maC\left(\maC(B_i)\right)> 1-\mu_\maS\left(\maS\left(A_i\right)\right).\label{eq:equivalence1}
\end{align}
\item[(ii)] If not, we show that
\begin{equation}
\label{eq:equivalence2}
I \subseteq I_i\mbox{ for some }i\in\llbracket 1,p \rrbracket.
\end{equation}
For this, fix $c,c' \in A$, and suppose that $k\in A_i$ and $k'\in A_{i'}$. 
Assume that $B \not\subseteq B_i$. Then, there exists a class of servers $s$ satisfying
\begin{equation*}
s \in B\subset \maS\left(A_{i'}\right)=\maS\left(\{c'\}\right),
\end{equation*}
from Lemma \ref{lemma:separable}. This implies that $(c',s)\in E$, so $I$ is not an independent set, an absurdity. Hence $B \subseteq B_i$ and similarly,
$B\subseteq B_{i'}$. As the latter sets are disjoints, this is true if and only if $i=i'$, which implies in turn that $\{c,c'\}\subseteq A_i$ for all couples $\{c,c'\}$ of elements
of $A$, so $A \subseteq A_i$, which concludes the proof of (\ref{eq:equivalence2}).

Finally, in view of Lemma \ref{lemma:separable}, (\ref{eq:equivalence2}) together with (\ref{eq:equivalence1}) applied to $I_i$ imply that
\begin{align*}
\mu_\maC\left(\maC(A)\right)+\mu_\maS\left(\maS(B)\right)
&=\mu_\maC\left(\maC(A_i)\right)+\mu_\maS\left(\maS(B_i)\right)\\
&>1 \ge 1- \mu\Bigl(E \cap \left(\maC_\circ\left(I\right)\times \maS_\circ\left(I\right)\right)\Bigl),
\end{align*}
which concludes the proof. 
\end{itemize}
\end{proof}

\subsection{Stationary ergodic framework}
\label{subsec:statergo}
Our coupling results will be more easily formulated in the ergodic theoretical framework. For this, we work on the canonical space $\Omega^0:=\left(F\times \mathbb S\times \mathbb C\right)^\mathbb Z$ of the input, 
on which we define the bijective shift operator $\theta$ by $\theta\left((\omega_n)_{n\in\mathbb Z}\right)= (\omega_{n+1})_{n\in\mathbb Z}$ for all $(\omega_n)_{n\in \mathbb Z} \in \Omega$. 
We denote by $\theta^{-1}$ the reciprocal operator of $\theta$, and by $\theta^n$ and $\theta^{-n}$ the $n$-th iterated of 
$\theta$ and $\theta^{-1}$, respectively, for all $n\in\N$. We equip $\Omega^0$ with a sigma-field $\mathscr F^0$ and with the image probability measure 
$\bp$ of the sequence $\left(C_n,S_n,\Sigma_n,\Gamma_n\right)_{n\in \Z}$ on $\Omega^0$. Observe that, under assumption (H1), $\bp$ is {compatible} with the shift, i.e. for any $\maA \in \mathscr F^0$, $\bpr{\maA}=\bpr{\theta^{-1}\maA}$ and any $\theta$-invariant event is either $\bp$-negligible or almost sure. 
Then the quadruple $\mathscr Q:=\left(\Omega^0,\mathscr F^0,\bp,\theta\right)$, termed {\em Palm space} of the input, is stationary ergodic.  
For more details about this framework, we refer the reader to the monographs \cite{BranFranLis90}, \cite{BacBre02} (Sections 2.1 and 2.5) and \cite{Rob03} (Chapter 7). 

Let the random variable (r.v. for short) $(\maC,\maS,\Sigma,\Gamma)$ be the projection of sample paths over their 0-coordinate. Thus $(\maC,\maS,\Sigma,\Gamma)$ can be interpreted as the input brought to the system at time 0, that is, 
at 0 a couple $(\maC,\maS)$ enters the systems, in which  $C$ has a list of preference $\Sigma$ over $\maS$ and $S$ has a list of preference $\Gamma$ over $\maC$. 
Then for any $n\in \mathbb Z$, the r.v. $(\maC,\maS,\Sigma,\Gamma)\circ\theta^n$ corresponds to the input brought to the system at time $n$.   
In what follows, for any $\CE$-valued r.v. $V$, we let $\suite{U^{[V]}_n}$ be the $\CE$-valued buffer detail sequence of the model, whenever the buffer detail at time 0 equals $V$.  
From (\ref{eq:defodot}), for any fixed bipartite matching structure and any fixed matching policy $\phi$, the sequence $\suite{U^{[V]}_n}$ is stochastic recursive, in that it obeys the recurrence relation 
\begin{equation}
\label{eq:recur}
\left\{\begin{array}{ll}
U^{[V]}_0 &= V\\
U^{[V]}_{n+1} &=  U^{[V]}_{n}\odot_{\phi} (\maC,\maS,\Sigma,\Gamma)\circ\theta^n,\,n\in\N
\end{array}\right.\quad ,\,\bp-\mbox{ a.s..}
\end{equation}

It follows from the stationarity of $\mathscr Q$ that a stationary buffer detail, if any, is a $\CE$-valued sequence $\suite{U_n}$ that is such that  
$U_n=U\circ\theta^n$, $\bp$-a.s. for all $n\in \N$, where $U$ is a $\CE$-valued r.v.. In turn, with (\ref{eq:recur}) this amounts to saying that $U$ is solution to the equation 
\begin{equation}
\label{eq:recurstat}
U\circ\theta = U\odot_{\phi} (\maC,\maS,\Sigma,\Gamma),\,\bp-\mbox{ a.s..}
\end{equation}
In other words, finding a stationary version of the buffer detail sequence amounts to solving the almost sure equation (\ref{eq:recurstat}). 
Further, such a solution corresponds uniquely to a stationary distribution of the buffer detail on the original probability space (see again the aforementioned references for details). 

By applying the very argument of the proof of Lemma 3.2 in \cite{BGM13} and Birkhoff's Theorem (instead of the SLLN), it is immediate to observe that (\ref{eq:Ncond}) is also necessary for the stability of the system 
under (H1). Specifically, there clear cannot exist a proper solution $U$ to (\ref{eq:recurstat}) such that $\bpr{U=\emptyset}>0$, unless (\ref{eq:Ncond}) holds.  

\medskip

To derive sufficient stability conditions and explicitly construct the equilibrium state of the system, the {\em backwards scheme} {\em \`a la Loynes} associated to the present recursion, is defined as follows: 
for any $n\in \N$, consider the r.v. $U^{[\emptyset]}_n \circ\theta^{-n}$, which can be interpreted as the buffer detail at time 0, 
starting from an empty system at time $-n$. The typical setting of Loynes's Theorem (see \cite{Loynes62}) is the case when the random map $x \mapsto x \odot_{\phi} (\maC,\maS,\Sigma,\Gamma)$ 
is almost surely continuous and monotonic, for a given metric and a partial ordering on $\CE$ such that $\CE$ admits a minimal point and all monotonic sequences 
of $\CE$ converge. Then an explicit solution of (\ref{eq:recurstat}) is obtained by taking the almost sure limit of the sequence $\suite{U^{[\emptyset]}_n \circ\theta^{-n}}$ - 
and the coupling of any sequence $\suite{U^{[V]}_n}$ to the stationary version, follows easily. 

In the present case, the recursion does not put in evidence any particular monotonicity property. 
However, we can use the sub-additivity of the matching policy under consideration, to obtain a coupling result in the strong backwards sense. For this we use Borovkov's theory of renovating events.  
Following \cite{Bor84}, we say that the buffer detail sequence $\suite{U^{[V]}_n}$ converges with {\em strong backwards coupling} to the stationary buffer detail sequence 
$\suite{U\circ\theta^n}$ if, $\bp$-almost surely, there exists $N^*\ge 0$ such that for all $n \ge N^*$, $U^{[V]}_n\circ\theta^{-n}=U$. 
Note that strong backwards coupling implies the forward coupling between $\suite{U^{[V]}_n}$ and $\suite{U\circ\theta^n}$, that is, there exists $\bp$-a.s. an integer $N\ge 0$ such that $U^{[V]}_{n}=U\circ\theta^n$ 
for all $n \ge N$. In particular the distribution of $U^{[V]}_{n}$ converges in total variation to that of $U$; see e.g. Section 2.4 of \cite{BacBre02} for details. 


\subsection{Renovating events}
\label{subsec:renove} 
In this Section we adapt the coupling results in section 4.4 of \cite{MBM17} to the case of EBM models (instead of GM models). 
For this, we define the following family of events for any $\CE_0$-valued r.v. $V$,  
\begin{equation*}
\masA_l(V) =\left\{U^{[V]}_{l}=\emptyset\right\},\,l\in\N.  
\end{equation*}
For any $k\in\N$ and any $l > -k$, the event 
\[\theta^k \masA_{l+k}(V)= \left\{U^{[V]}_{l+k}\circ\theta^{-k}=\emptyset\right\}\]
has the following interpretation: a model started in state $V$ at time $-k$ is empty at time $l$. 
Thus if we denote $V=(W,Z)$ we have that 
\[\theta^k \masA_{l+k}\left((W,Z)\right) = \left\{Q_\phi\Bigl(WC\circ\theta^{-k}... \,\,C\circ\theta^{l-1}\,,\,ZS\circ\theta^{-k}...\,\,S\circ\theta^{l-1} \Bigl)=\emptyset\right\}.\]
Clearly, $\suite{\masA_n(V)}$ form a sequence of renovating events of length 1 for the recursion $\suite{U^{[V]}_n}$, for any such initial condition $V$ (see \cite{Foss92,Foss94}).  
Thus the following result is a consequence e.g. of Theorem 2.5.3 of \cite{BacBre02}, and is proven similarly to Proposition 8 in 
\cite{MBM17}, 

\begin{proposition}
\label{pro:renov1}
Let $G=(\maV,\maE)$ be a matching graph, $\phi$ be an admissible policy and 
$V$ be a $\CE$-valued random variable. Suppose that assumption (H1) holds, and that    
\begin{equation}
\label{eq:renov0}
\lim_{n\to\infty} \bpr{\bigcap_{k=0}^{\infty} \bigcup_{l=0}^n \masA_{l}(V) \cap \theta^k\masA_{l+k}(V)}=1. 
\end{equation}
Then, there exists a stationary buffer detail sequence $\suite{U\circ\theta^n}$, to which  
$\suite{U^{[V]}_{n}}$ converges with strong backwards coupling. Moreover we have $\bpr{U=\emptyset}>0$. 
\end{proposition}


Define the following sets of random variables: 
\begin{equation*}
\mathscr V^r=\Bigl\{\CE_0-\mbox{ valued r.v. $(W,Z)$:\, $|W|=|Z| \le r$ a.s.}\Bigl\},\,r\in \N_+,
\end{equation*}
and let 
\[\mathscr V^{\infty}:= \bigcup_{r=1}^{+\infty} \mathscr V^r.\]
Define the following event for any admissible input $(\mc,\ms) \in \maC^*\times \maS^*$ of length $m$, 
\begin{align}
\mathscr B\left(\mc\,,\,\ms\right) 
&=\left\{\left(C\,C\circ\theta \,... \,C\circ\theta^{m-1}\;,\;S\,S\circ\theta \,... \,S \circ\theta^{m-1}\right)=\left(\mc\,,\,\ms\right)\right\}.\label{eq:defB}
\end{align}

We have the following analog to Theorem 3 in \cite{MBM17}, 
\begin{theorem}
\label{thm:main}
Let $\phi$ be a sub-additive policy and $\mathcal B$ be a bipartite matching structure such that $(\mathcal B,\phi)$ admit at least one 
strong erasing couple. Suppose that for some $r \in \mathbb N_+$. 
\begin{equation}
\label{eq:renov1}
\lim_{n\to\infty} \bpr{\bigcap_{k=0}^{\infty} \bigcup_{l=0}^n \masA_l(\emptyset) \cap \theta^k\masA_{l+k}(\emptyset) \cap \theta^{-l}\mathscr B\left(\mc^1...\mc^r\,,\,\ms^1\,...\,\ms^r\right)}=1, 
\end{equation}
where $(\mc^1,\mc^1),\,...,\,(\mc^r,\mc^r)$ are $r$ (possibly identical) strong erasing couples for $\mathcal B$ and $\phi$.  
Then, there exists a solution $U^r$ to (\ref{eq:recurstat}) in $\mathscr V^{\infty}$ such that $\bpr{U^r=\emptyset}>0$, and  
to which all sequences $\suite{U^{[V]}_{n}}$, for $V\in \mathscr V^{r}$, converge with strong backwards coupling. 

If the above is true for any $r \in \N_+$, then there exists a solution $U^*$ to (\ref{eq:recurstat}) such that $\bpr{U^*=\emptyset}>0$, 
and to which all sequences $\suite{U^{[V]}_{n}}$, for $V\in \mathscr V^{\infty}$, converge with strong backwards coupling.
\end{theorem}

\begin{proof}
Understanding (strong) erasing couples as the analog of (strong) erasing words in \cite{MBM17}, we can prove the first statement for all $r$ similarly to Proposition 9 in \cite{MBM17}. 
This follows, as in Lemma 5 in \cite{MBM17}, from the fact that 
$\left\{\masA_n\left(\emptyset\right) \circ\theta^{-n}\mathscr B(\left(\mc^1...\mc^r\,,\,\ms^1\,...\,\ms^r\right)\right\}_{\N}$ is a sequence of renovating events of length $m=\sum_{i=1}^r |\mc^i|$ 
for the recursion $\suite{U^{[V]}_{n}}$, for any $V\in \mathscr V^{r}$. The uniqueness statement follows by recurrence on $r$, 
exactly as in the proof of Theorem 3 in \cite{MBM17}. 
\end{proof}

\begin{ex}
\label{ex:NNbis}
\rm
We address the toy example of the 'NN' graph, for which we explicitly construct the unique solution to (\ref{eq:recurstat}) for $\phi=\textsc{fcfs}$ and {\sc lcfs}.
   
Specifically, consider the structure $\maB=\left(\{1,2,3\},\{\bar 1,\bar 2,\bar 3\},E,F\right)$, where the matching graph 
$E$ and the arrival graph $F$ are respectively given by the left and right graphs in Figure \ref{Fig:NNarrival}. 
\begin{figure}[h!]
\begin{center}
\begin{tikzpicture}
\draw[-] (-3.5,3) -- (-2.5,2);  
\draw[-] (-3.5,3) -- (-3.5,2);
\draw[-] (-2.5,3) -- (-1.5,2);
\draw[-] (-2.5,3)-- (-2.5,2);
\draw[-] (-1.5,3)-- (-1.5,2); 
\fill (-3.5,3) circle (2.5pt) node[above] {\small{1}} ;
\fill (-2.5,3) circle (2.5pt) node[above] {\small{2}} ;
\fill (-1.5,3) circle (2.5pt) node[above] {\small{3}} ;
\fill (-3.5,2) circle (2.5pt) node[below] {\small{$\bar 1$}} ;
\fill (-2.5,2) circle (2.5pt) node[below] {\small{$\bar 2$}} ;
\fill (-1.5,2) circle (2.5pt) node[below] {\small{$\bar 3$}} ;
\draw[-] (1.5,3) -- (2.5,2);  
\draw[-] (2.5,3) -- (1.5,2);
\draw[-] (2.5,3) -- (3.5,2);
\draw[-] (3.5,3)-- (2.5,2);
\draw[-] (3.5,3)-- (1.5,2); 
\fill (1.5,3) circle (2.5pt) node[above] {\small{1}} ;
\fill (2.5,3) circle (2.5pt) node[above] {\small{2}} ;
\fill (3.5,3) circle (2.5pt) node[above] {\small{3}} ;
\fill (1.5,2) circle (2.5pt) node[below] {\small{$\bar 1$}} ;
\fill (2.5,2) circle (2.5pt) node[below] {\small{$\bar 2$}} ;
\fill (3.5,2) circle (2.5pt) node[below] {\small{$\bar 3$}} ;
\end{tikzpicture}
\caption[smallcaption]{Matching graph (left) and arrival graph (right) of Example \ref{ex:NNbis}.} \label{Fig:NNarrival}
\end{center}
\end{figure}
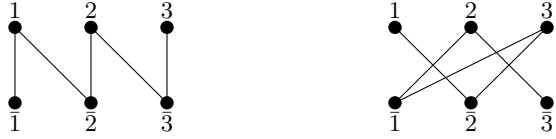

The matching structure $\maB$ satisfies case (iiia) of Proposition \ref{pro:strongcouple}, for $\check{\maC}=\{1,2\}$, $\check{\maS}=\{\bar 2,\bar 3\}$ and 
$\mathscr P=1\v \bar 2 \v 2 \v \bar 3$. Thus there necessarily exist strong erasing couples for both {\sc fcfs} and {\sc lcfs}. 
The construction below is independent of the lists of preferences, and we drop again these parameters from all notation. 
Set $\Omega^0:=\left\{\om,\,\theta\om,\,...\,,\theta^8\om\right\}$, where 
\begin{equation}
\label{eq:om}
\small{\om=...(1,\bar 2)(2,\bar 1)(1,\bar 2)(2,\bar 3)(1,\bar 2)(2,\bar 3)(2,\bar 3)(3,\bar 1)(3,\bar 2)\underline{\mathbf{(1,\bar 2)}}(2,\bar 1)(1,\bar 2)(2,\bar 3)(1,\bar 2)(2,\bar 3)(2,\bar 3)(3,\bar 1)(3,\bar 2)...,}
\end{equation}
in which the underlined couple is the $0$-coordinate (i.e. the origin of time). Setting $\mathscr F^0$ as the power-set of $\Omega^0$  
and $\bp$ the uniform probability on $\Omega^0$, it is immediate that $\mathscr Q^0=\left(\Omega^0,\mathscr F^0,\bp,\theta\right)$ is a stationary ergodic quadruple. 
Observe that the image measure $\bp$ corresponds to the following probability measure $\mu$ on $F$: 
\[\mu(1,\bar 2)={1 \over 3},\quad\mu(2,\bar 3)={1 \over 3},\quad\mu(2,\bar 1)={1 \over 9},\quad\mu(3,\bar 2)={1 \over 9},\quad\mu(3,\bar 1)={1 \over 9},\] 
which clearly satisfies (\ref{eq:Ncond}). (Observe however, taking the independent set $I=\{3\}\times\{\bar 1\}$, that (\ref{eq:Scond}) does not hold.) 
Consider the couple $(\mc,\ms)=(331212122,\bar 1\bar 2\bar 2\bar 1 \bar 2\bar 3\bar 2\bar 3\bar 3)$. It can be checked that 
$Q_{\textsc{fcfs}}\left(\breve{\mc},\breve{\ms}\right)=\emptyset$ and $Q_{\textsc{lcfs}}\left(\breve{\mc},\breve{\ms}\right)=\emptyset$ for any respective suffixes $\breve{\mc}$ and $\breve{\ms}$ of $\mc$ and $\ms$. 
Moreover, we obtain on the one hand that 
\[Q_{\textsc{fcfs}}\left(3{\mc},\bar 1{\ms}\right)=Q_{\textsc{fcfs}}\left(3{\mc},\bar 2{\ms}\right)=Q_{\textsc{fcfs}}\left(2{\mc},\bar 1{\ms}\right)
=Q_{\textsc{fcfs}}\left(1{\mc},\bar 3{\ms}\right)=\emptyset,\]
and on the other hand, 
\[Q_{\textsc{lcfs}}\left(3{\mc}\mc\,\bar 1{\ms}\ms\right)=Q_{\textsc{fcfs}}\left(3{\mc}\mc,\bar 2{\ms}\ms\right)=Q_{\textsc{fcfs}}\left(2{\mc}\mc,\bar 1{\ms}\ms\right)
=Q_{\textsc{fcfs}}\left(1{\mc}\mc,\bar 3{\ms}\ms\right)=\emptyset.\]
Thus $(\mc,\ms)$ is a strong erasing couple for $(\maB,\textsc{fcfs})$, whereas $(\mc\mc,\ms\ms)$ is a strong erasing couple for $(\maB,\textsc{lcfs})$. 
By the very definition of $\Omega^0$, condition (\ref{eq:renov1}) is trivially satisfied for all $r$ for both {\sc fcfs} and {\sc lcfs}. Therefore there exists a unique stationary 
solution $U^{\sf}$ (resp., $U^{\sl}$) of (\ref{eq:recurstat}) for {\sc fcfs} (resp., {\sc lcfs}). They are respectively given by 
\begin{equation}
\label{eq:solNNfcfs}\left\{\begin{array}{c}
U^{\sf}(\om)=(33,\bar 1\bar 2),\,U^{\sf}(\theta\om)=(33,\bar 2\bar 2),\,U^{\sf}(\theta^2\om)=(33,\bar 2\bar 1),\,U^{\sf}(\theta^3\om) =(33,\bar 1\bar 2),\\U^{\sf}(\theta^4\om)=(3,\bar 1),\,
U^{\sf}(\theta^5\om)=(3,\bar 2),\,U^{\sf}(\theta^6\om)=\emptyset,\,U^{\sf}(\theta^7\om)=\emptyset,\,U^{\sf}(\theta^8\om)=(3,\bar 1);
\end{array}\right.
\end{equation}
\begin{equation}
\label{eq:solNNlcfs}\left\{\begin{array}{c}
U^{\sl}(\om)=(33,\bar 1\bar 2),\,U^{\sl}(\theta\om)=(33,\bar 1\bar 2),\,U^{\sl}(\theta^2\om)=(33,\bar 1\bar 1),\,U^{\sl}(\theta^3\om) =(33,\bar 1\bar 2),\\U^{\sl}(\theta^4\om)=(3,\bar 1),\,
U^{\sl}(\theta^5\om)=(3,\bar 2),\,U^{\sl}(\theta^6\om)=\emptyset,\,U^{\sl}(\theta^7\om)=\emptyset,\,U^{\sl}(\theta^8\om)=(3,\bar 1).
\end{array}\right.
\end{equation}
\end{ex}

\subsection{Independent Case}
\label{subsec:iid}
We show in this Section, that the conditions (\ref{eq:renov0}) and (\ref{eq:renov1}) take a simple form 
if we assume additionally that the input sequence is mutually independent, i.e. under assumption (IID)). 

Denote for any $\CE_0$-valued r.v. $V=(W,Z)$, for any $k\in\mathbb N^*$, by $\tau_j(V)$ the $j$-th visit time to $\emptyset$ for the process $\left(U_n^{[V]}\right)_n$:
\[\tau_0(V) :=0, \quad \tau_j(V) := \inf \{n > \tau_{j-1}(V), U_{n}^{[V]} = \emptyset \}, \; j\geq 1.\] 
We define the following stability condition depending on the initial condition $V$,  
\begin{itemize}
\item[\textbf{(H2)}] The stopping time $\tau_1(V)$ is integrable.
\end{itemize} 
A quick survey of the literature on the stability of matching models gives the following, 
\begin{proposition}
\label{prop:H2}
Assumption (H2) holds true for any $\CE_0$-valued r.v. $V$, in the following cases: 
\begin{enumerate}
\item $(\maC\cup\maS,A)$ is strongly connected and $(\maB,\mu)$ satisfies (\ref{eq:Scond});
\item $(\maC\cup\maS,A)$ is strongly connected, $(\maB,\mu)$ satisfies (\ref{eq:Ncond}) and $\phi=\textsc{ml}$;
\item $(\maC\cup\maS,A)$ is strongly connected, $\maG$ is bi-separable and  $(\maB,\mu)$ satisfies (\ref{eq:scondmonotone}); 
\item $\maB$ defines a BM model and $\phi=\textsc{fcfs}$;
\item $\maB$ defines a GM model and $\phi=\textsc{fcfs}$; 
\item $\maB$ defines a GM model and $\phi=${\sc ml}. 
\end{enumerate}
\end{proposition}

\begin{proof}
Assertions 1 and 2 follow from Theorem 4.2 in \cite{BGM13}, and respectively from Proposition 5.2 and Theorem 7.1 in \cite{BGM13}. Assertion 
4 follows from Theorem 3 in \cite{AW11} (revisited in Theorem 2 of \cite{ABMW17}). Items 5 and 6 are precisely assertions 1 and 2 of Proposition 11 in \cite{MBM17}. 
To prove assertion 3, just notice that (\ref{eq:scondmonotone}) entails (\ref{eq:Scond}) from Proposition \ref{pro:equivalenceBGM}, and we conclude applying 1. 
\end{proof}

We have the following result, 
\begin{theorem}
\label{thm:mainiid}
Suppose that assumptions (IID) and (H2) hold. If $\phi$ is sub-additive and $(\mathcal B,\phi)$ are such that any admissible buffer detail $(\mw,\mz)$ admits an erasing couple, then  
there exists a unique solution $U^*$ to (\ref{eq:recurstat}) in $\mathscr V^{\infty}$ such that $\bpr{U^*=\emptyset}>0$, and to which all sequences 
$\suite{U^{[V]}_n}$, for $V \in \mathscr V^{\infty}$, converge with strong backwards coupling. 
\end{theorem} 

\begin{proof}
This result is analog to Theorem 3 in \cite{MBM17}, and we skip the technical details of the proof. The main steps, which can be checked using similar arguments to 
those in \cite{MBM17}, are as follows: 
\begin{enumerate}
\item Under these assumptions, an analog argument to Proposition 13 in \cite{MBM17} shows that there is forward coupling between $\suite{U^{[V]}_n}$ and $\suite{U^{[V^*]}_n}$ for any $V$ and $V^*$ in $\mathscr V^\infty$;
\item Fix a r.v. $V\in\mathscr V^{\infty}$, and let $r$ be such that $V \in \mathscr V^r$. From the independence of the input, the sub-additivity of $\phi$ and the existence of erasing couples of any $(\mw,\mz)$ 
for $(\mathcal B,\phi)$, one can show similarly to Proposition 12 in \cite{MBM17} that assumption (H2) entails (\ref{eq:renov0}); 
\item As (H1) holds true under the iid assumptions, we can apply Proposition \ref{pro:renov1}: $V$ converges with strong backwards coupling, and thereby also in the forward sense, to a stationary sequence 
$\suite{U\circ\theta^n}$, where $U\in\mathscr V^{\infty}$. But from 1., any pair of such stationary sequences $\suite{U\circ\theta^n}$ and $\suite{U^*\circ\theta^n}$ couple, and therefore coincide almost surely. 
Thus there exists a unique solution $U$ to (\ref{eq:recurstat}) in $\mathscr V^{\infty}$. This completes the proof. 
\end{enumerate}  
\end{proof}
\paragraph{Summarizing the results.}
To summarize Theorems \ref{thm:main} and \ref{thm:mainiid}, for a given bipartite matching structure $\mathcal B$ and a sub-additive matching policy $\phi$ (such as 
{\sc ml}, random, priorities, {\sc lcfs} and {\sc fcfs} - Section \ref{sec:subadd}), there exists a unique proper stationary buffer detail on the Palm space $\mathscr Q$ in the two following general cases:
\begin{itemize} 
\item The input is stationary ergodic, $(\mathcal B,\phi)$ admits a strong erasing couple (which is true if Proposition \ref{pro:strongcouple} - or Proposition \ref{pro:strongcoupleBM} for a BM model - holds true), and the renovation condition (\ref{eq:renov1}) (which entails (\ref{eq:Ncond})) holds true for any $r \ge 1$.  
\item The input is iid, any admissible state admits an erasing couple for $(\mathcal B,\phi)$ (which is true n particular under either conditions of Proposition \ref{prop:existerasing}), 
and the regeneration condition (H2) (satisfied under either conditions of Proposition \ref{prop:H2}) holds true.  
\end{itemize}

\subsection{Constructing bi-infinite perfect matchings}
\label{subsec:matchings} 
Fix a bipartite matching structure $\mathcal B$ and a sub-additive matching policy $\phi$. Suppose that we are under either one of the above conditions. Then there exists, on the Palm space $\mathscr Q$ of the input, a unique $\theta$-compatible buffer-content sequence $\suite{U\circ\theta^n}$, that is such that $\bpr{U=\emptyset}>0$. By the ergodicity of the shift, this readily entails that the set 
\[\left\{n \in \mathbb Z\,:\,U\circ\theta^n = \emptyset\right\}=:\left\{\mbox{Construction points of the system on }\mathbb Z\right\}\]
is infinite (and infinite on both sides of the origin). As is done for stationary queuing systems e.g. in \cite{Nev83} or Chapter 2 of \cite{BacBre02}, and for stationary BM models in Section 4.6 of \cite{MBM17}, 
we can easily use these construction points to build a unique stationary {\em matching} of the bi-infinite input $\left((\maC,\maS,\Sigma,\Gamma)\circ\theta^n\right)_{n\in\mathbb Z}$, by just constructing the unique matching of the incoming customers and servers between each successive construction points. 

\begin{ex}[Example \ref{ex:NNbis}, continued]
\rm
From the unique respective stationary solutions of (\ref{eq:recurstat}) for $\phi=\textsc{fcfs}$ and {\sc lcfs}, given in (\ref{eq:solNNfcfs}) and (\ref{eq:solNNlcfs}), we can construct 
the unique {\sc fcfs} and the unique {\sc lcfs} matchings of the input considered in Example \ref{ex:NNbis}. They are respectively represented (for the sample $\om$ defined by (\ref{eq:om})), in Figure 
\ref{Fig:matchingNN}.  

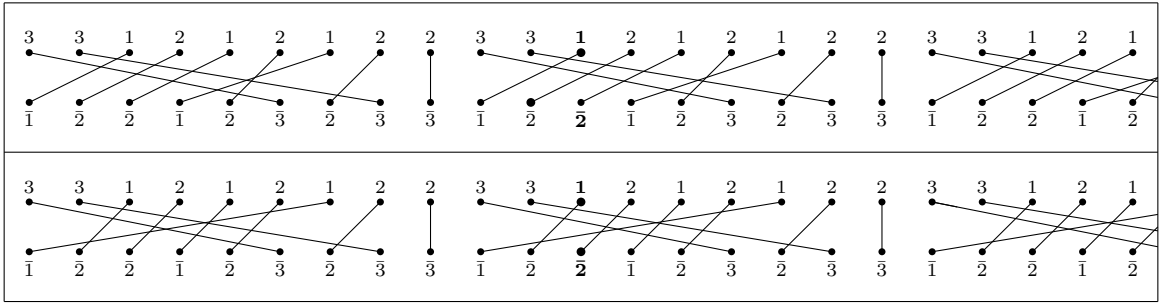
\begin{figure}[h!]
\begin{center}
\begin{tikzpicture}[scale=0.66]
\draw[-] (0.5,1) -- (23.5,1);
\draw[-] (0.5,1) -- (0.5,-5);
\draw[-] (0.5,-2) -- (23.5,-2);
\draw[-] (0.5,-5) -- (23.5,-5);
\draw[-] (23.5,-5) -- (23.5,1);
\fill (1,0) circle (2pt) node[above] {\scriptsize{3}} ;
\fill (2,0) circle (2pt) node[above] {\scriptsize{3}} ;
\fill (3,0) circle (2pt) node[above] {\scriptsize{1}} ;
\fill (4,0) circle (2pt) node[above] {\scriptsize{2}} ;
\fill (5,0) circle (2pt) node[above] {\scriptsize{1}} ;
\fill (6,0) circle (2pt) node[above] {\scriptsize{2}} ;
\fill (7,0) circle (2pt) node[above] {\scriptsize{1}} ;
\fill (8,0) circle (2pt) node[above] {\scriptsize{2}} ;
\fill (9,0) circle (2pt) node[above] {\scriptsize{2}} ;
\fill (10,0) circle (2pt) node[above] {\scriptsize{3}} ;
\fill (11,0) circle (2pt) node[above] {\scriptsize{3}} ;
\fill (12,0) circle (2.5pt) node[above] {\scriptsize{$\mathbf{1}$}} ;
\fill (13,0) circle (2pt) node[above] {\scriptsize{2}} ;
\fill (14,0) circle (2pt) node[above] {\scriptsize{1}} ;
\fill (15,0) circle (2pt) node[above] {\scriptsize{2}} ;
\fill (16,0) circle (2pt) node[above] {\scriptsize{1}} ;
\fill (17,0) circle (2pt) node[above] {\scriptsize{2}} ;
\fill (18,0) circle (2pt) node[above] {\scriptsize{2}} ;
\fill (19,0) circle (2pt) node[above] {\scriptsize{3}} ;
\fill (20,0) circle (2pt) node[above] {\scriptsize{3}} ;
\fill (21,0) circle (2pt) node[above] {\scriptsize{1}} ;
\fill (22,0) circle (2pt) node[above] {\scriptsize{2}} ;
\fill (23,0) circle (2pt) node[above] {\scriptsize{1}} ;
\fill (1,-1) circle (2pt) node[below] {\scriptsize{$\bar 1$}} ;
\fill (2,-1) circle (2pt) node[below] {\scriptsize{$\bar 2$}} ;
\fill (3,-1) circle (2pt) node[below] {\scriptsize{$\bar 2$}} ;
\fill (4,-1) circle (2pt) node[below] {\scriptsize{$\bar 1$}} ;
\fill (5,-1) circle (2pt) node[below] {\scriptsize{$\bar 2$}} ;
\fill (6,-1) circle (2pt) node[below] {\scriptsize{$\bar 3$}} ;
\fill (7,-1) circle (2pt) node[below] {\scriptsize{$\bar 2$}} ;
\fill (8,-1) circle (2pt) node[below] {\scriptsize{$\bar 3$}} ;
\fill (9,-1) circle (2pt) node[below] {\scriptsize{$\bar 3$}} ;
\fill (10,-1) circle (2pt) node[below] {\scriptsize{${\bar 1}$}} ;
\fill (11,-1) circle (2.5pt) node[below] {\scriptsize{$\bar 2$}} ;
\fill (12,-1) circle (2pt) node[below] {\scriptsize{$\mathbf{\bar 2}$}} ;
\fill (13,-1) circle (2pt) node[below] {\scriptsize{$\bar 1$}} ;
\fill (14,-1) circle (2pt) node[below] {\scriptsize{$\bar 2$}} ;
\fill (15,-1) circle (2pt) node[below] {\scriptsize{$\bar 3$}} ;
\fill (16,-1) circle (2pt) node[below] {\scriptsize{$\bar 2$}} ;
\fill (17,-1) circle (2pt) node[below] {\scriptsize{$\bar 3$}} ;
\fill (18,-1) circle (2pt) node[below] {\scriptsize{$\bar 3$}} ;
\fill (19,-1) circle (2pt) node[below] {\scriptsize{$\bar 1$}} ; 
\fill (20,-1) circle (2pt) node[below] {\scriptsize{$\bar 2$}} ;
\fill (21,-1) circle (2pt) node[below] {\scriptsize{$\bar 2$}} ;
\fill (22,-1) circle (2pt) node[below] {\scriptsize{$\bar 1$}} ;
\fill (23,-1) circle (2pt) node[below] {\scriptsize{$\bar 2$}} ; 
\draw[-] (1,0) -- (6,-1);
\draw[-] (2,0) -- (8,-1);
\draw[-] (3,0) -- (1,-1);
\draw[-] (4,0) -- (2,-1); 
\draw[-] (5,0) -- (3,-1);
\draw[-] (6,0) -- (5,-1);
\draw[-] (7,0) -- (4,-1);
\draw[-] (8,0) -- (7,-1);
\draw[-] (9,0) -- (9,-1);
\draw[-] (10,0) -- (15,-1);
\draw[-] (11,0) -- (17,-1);
\draw[-] (12,0) -- (10,-1);
\draw[-] (13,0) -- (11,-1);
\draw[-] (14,0) -- (12,-1);
\draw[-] (15,0) -- (14,-1);
\draw[-] (16,0) -- (13,-1);
\draw[-] (17,0) -- (16,-1);
\draw[-] (18,0) -- (18,-1);
\draw[-] (19,0) -- (23.5,-0.9);
\draw[-] (20,0) -- (23.5,-0.58);
\draw[-] (21,0) -- (19,-1);
\draw[-] (22,0) -- (20,-1); 
\draw[-] (23,0) -- (21,-1);
\draw[-] (22,-1) -- (23.5,-0.5);
\draw[-] (23,-1) -- (23.5,-0.5);
%
%
\fill (1,-3) circle (2pt) node[above] {\scriptsize{3}} ;
\fill (2,-3) circle (2pt) node[above] {\scriptsize{3}} ;
\fill (3,-3) circle (2pt) node[above] {\scriptsize{1}} ;
\fill (4,-3) circle (2pt) node[above] {\scriptsize{2}} ;
\fill (5,-3) circle (2pt) node[above] {\scriptsize{1}} ;
\fill (6,-3) circle (2pt) node[above] {\scriptsize{2}} ;
\fill (7,-3) circle (2pt) node[above] {\scriptsize{1}} ;
\fill (8,-3) circle (2pt) node[above] {\scriptsize{2}} ;
\fill (9,-3) circle (2pt) node[above] {\scriptsize{2}} ;
\fill (10,-3) circle (2pt) node[above] {\scriptsize{3}} ;
\fill (11,-3) circle (2pt) node[above] {\scriptsize{3}} ;
\fill (12,-3) circle (2.5pt) node[above] {\scriptsize{$\mathbf{1}$}} ;
\fill (13,-3) circle (2pt) node[above] {\scriptsize{2}} ;
\fill (14,-3) circle (2pt) node[above] {\scriptsize{1}} ;
\fill (15,-3) circle (2pt) node[above] {\scriptsize{2}} ;
\fill (16,-3) circle (2pt) node[above] {\scriptsize{1}} ;
\fill (17,-3) circle (2pt) node[above] {\scriptsize{2}} ;
\fill (18,-3) circle (2pt) node[above] {\scriptsize{2}} ;
\fill (19,-3) circle (2pt) node[above] {\scriptsize{3}} ;
\fill (20,-3) circle (2pt) node[above] {\scriptsize{3}} ;
\fill (21,-3) circle (2pt) node[above] {\scriptsize{1}} ;
\fill (22,-3) circle (2pt) node[above] {\scriptsize{2}} ;
\fill (23,-3) circle (2pt) node[above] {\scriptsize{1}} ;
\fill (1,-4) circle (2pt) node[below] {\scriptsize{$\bar 1$}} ;
\fill (2,-4) circle (2pt) node[below] {\scriptsize{$\bar 2$}} ;
\fill (3,-4) circle (2pt) node[below] {\scriptsize{$\bar 2$}} ;
\fill (4,-4) circle (2pt) node[below] {\scriptsize{$\bar 1$}} ;
\fill (5,-4) circle (2pt) node[below] {\scriptsize{$\bar 2$}} ;
\fill (6,-4) circle (2pt) node[below] {\scriptsize{$\bar 3$}} ;
\fill (7,-4) circle (2pt) node[below] {\scriptsize{$\bar 2$}} ;
\fill (8,-4) circle (2pt) node[below] {\scriptsize{$\bar 3$}} ;
\fill (9,-4) circle (2pt) node[below] {\scriptsize{$\bar 3$}} ;
\fill (10,-4) circle (2pt) node[below] {\scriptsize{$\bar 1$}} ;
\fill (11,-4) circle (2pt) node[below] {\scriptsize{$\bar 2$}} ;
\fill (12,-4) circle (2.5pt) node[below] {\scriptsize{$\mathbf{\bar 2}$}} ;
\fill (13,-4) circle (2pt) node[below] {\scriptsize{$\bar 1$}} ;
\fill (14,-4) circle (2pt) node[below] {\scriptsize{$\bar 2$}} ;
\fill (15,-4) circle (2pt) node[below] {\scriptsize{$\bar 3$}} ;
\fill (16,-4) circle (2pt) node[below] {\scriptsize{$\bar 2$}} ;
\fill (17,-4) circle (2pt) node[below] {\scriptsize{$\bar 3$}} ;
\fill (18,-4) circle (2pt) node[below] {\scriptsize{$\bar 3$}} ;
\fill (19,-4) circle (2pt) node[below] {\scriptsize{$\bar 1$}} ;
\fill (20,-4) circle (2pt) node[below] {\scriptsize{$\bar 2$}} ;
\fill (21,-4) circle (2pt) node[below] {\scriptsize{$\bar 2$}} ;
\fill (22,-4) circle (2pt) node[below] {\scriptsize{$\bar 1$}} ;
\fill (23,-4) circle (2pt) node[below] {\scriptsize{$\bar 2$}} ; 
\draw[-] (1,-3) -- (6,-4);
\draw[-] (2,-3) -- (8,-4);
\draw[-] (3,-3) -- (2,-4);
\draw[-] (4,-3) -- (3,-4); 
\draw[-] (5,-3) -- (4,-4);
\draw[-] (6,-3) -- (5,-4);
\draw[-] (7,-3) -- (1,-4);
\draw[-] (8,-3) -- (7,-4);
\draw[-] (9,-3) -- (9,-4);
\draw[-] (10,-3) -- (15,-4);
\draw[-] (11,-3) -- (17,-4);
\draw[-] (12,-3) -- (11,-4);
\draw[-] (13,-3) -- (12,-4);
\draw[-] (14,-3) -- (13,-4);
\draw[-] (15,-3) -- (14,-4);
\draw[-] (16,-3) -- (10,-4);
\draw[-] (17,-3) -- (16,-4);
\draw[-] (18,-3) -- (18,-4);
\draw[-] (19,-3) -- (19.5,-3.1);
\draw[-] (19,-3) -- (23.5,-3.9);
\draw[-] (20,-3) -- (23.5,-3.58);
\draw[-] (21,-3) -- (20,-4);
\draw[-] (22,-3) -- (21,-4); 
\draw[-] (23,-3) -- (22,-4);
\draw[-] (19,-4) -- (23.5,-3.25);
\draw[-] (23,-4) -- (23.5,-3.5);
\end{tikzpicture}
\caption[smallcaption]{Unique bi-infinite {\sc fcfs} matching (top) and {\sc lcfs} matching (bottom) for the matching and arrival graphs of Figure \ref{Fig:NNarrival} and the sample $\om$ in (\ref{eq:om}). 
(The coordinate in bold is the origin of time.)} \label{Fig:matchingNN}
\end{center}
\end{figure}

\end{ex}

\appendix

\section{Proof of Proposition \ref{prop:nonexp1}}
\label{sec:proof}
Consider a random policy $\phi$, and the corresponding $\nu_\phi$ on $\mathbb S$ and $\rho_\phi$ on $\mathbb C$. 
Hereafter we show the non-expansiveness of $\phi$ {\em via} (\ref{eq:defnonexp1}), generalizing the argument of Lemma 4 in \cite{MoyPer17} to the EBM. 
First observe the following: for any state $(x,y)$, any incoming couple $(c,s)$ and any arrays $(\sigma,\gamma)$ drawn from 
 $\nu^\phi$, as the orders of preference of $c$ and $s$ are fixed by $a$ and $b$, the matching policy satisfies the following 
consistency property:
\begin{equation}
\label{eq:consist}
\left\{\begin{array}{ll}
\{p_{\phi}(y,c,\sigma,\gamma),p_{\phi}(y',c,\sigma,\gamma)\} \subset \maP(y,c) \cap \maP(y',c) &\Longrightarrow p_{\phi}(y,c,\sigma,\gamma) = p_{\phi}(y',c,\sigma,\gamma);\\
\{q_{\phi}(x,s,\sigma,\gamma),q_{\phi}(x',s,\sigma,\gamma)\} \subset \maQ(x,s) \cap \maQ(x',s) &\Longrightarrow q_{\phi}(x,s,\sigma,\gamma) = q_{\phi}(x',s,\sigma,\gamma),
\end{array}\right.       
\end{equation}
in other words, the choice of match of $c$ cannot be different in the two systems, if both options were available in both systems (and likewise for $s$). 
\medskip

Consider a system in the state $(x,y)$. Let $(c,s)\in F$ be the entering couple and $(\sigma,\gamma) \in \mathbb S\times\mathbb C$ be the 
couple of arrays of permutations respectively drawn from $\nu_\phi$ and $\rho_\phi$. 
We let $(\ell,k)$ (resp., $(\ell',k')$) be the classes of the respective
matches of the entering items $(c,s)$ if the class detail of the
system is $(x,y)$ (resp., $(x',y')$). We understand that
$k=\emptyset$ if $s$ does not find a match in the system described
by $(x,y)$, nor with $c$, and likewise for $\ell,k'$ and $\ell'$. 
By definition, we thus have $\ell=p_{\phi}(y,c,\sigma,\gamma)$ if $\maP(y,c)\ne \emptyset$, $k=q_{\phi}(x,s,\sigma,\gamma)$ 
if $\maQ(x,s)\ne \emptyset$, $\ell'=p_{\phi}(y',c,\sigma,\gamma)$ if $\maP(y',c)\ne \emptyset$, and $k'=p_{\phi}(x',s,\sigma,\gamma)$ if $\maQ(x',s)\ne \emptyset$. 
Also, for notational brevity we denote
\[(\tx,\ty)=(x,y)\ccc_{\phi}(c,s,\sigma,\gamma);\,\quad\,(\tx',\ty')=(x',y')\ccc_{\phi}(c,s,\sigma,\gamma).\]

$\underline{\mbox{Case 1: }s\not\in \maS(c).}$

\medskip

We then have the following sub-cases,
\begin{enumerate}
\item[(1a)] $\underline{k=k'\ne \emptyset,\ell=\ell\ne \emptyset'}.$ The incoming customer and
server are matched with items of the same respective classes. Thus
$\tx(i)=x(i)$ and $\tx'(i)=x'(i)$ for all $i \ne k$, $\ty(j)=y(j)$
and $\ty'(j)=y'(j)$ for all $j \ne \ell$, whereas
\[\tx(k)=x(k)-1,\,\ty(\ell)=y(\ell)-1,\,\tx'(k)=x'(k)-1,\,\ty'(\ell)=y'(\ell)-1.\]
Hence we have
\begin{align*}
\|(\tx',\ty') - (\tx,\ty)\| & =\sum_{i\ne k} |x(i)-x'(i)| +\sum_{j
\ne
\ell} |y(j)-y'(j)|\\
&\quad\quad+|x(k)-1-(x'(k)-1)|+|(y(\ell)-1)-(y'(\ell)-1)|\\
& = \sum_{i} |x(i)-x'(i)| +\sum_{j} |y(j)-y'(j)| =\|(x',y')-(x,y)\|.
\end{align*}
\item[(1b)] $\underline{k\ne \emptyset, k'\ne \emptyset, k\ne k',\ell=\ell'\ne\emptyset}.$
In this case, the new $c$ is matched with a server of the same class
$\ell$ in both systems, but $s$ is matched with customers of two
different classes in the two systems. Then, from (\ref{eq:consist}), 
we have either $k \not\in \maQ(x',s)$ or $k'\not\in \maQ(x,s)$ (or both). 
Suppose that $k'\not\in \maQ(x,s)$ that is, $x(k')=0$ (the case $x'(k)=0$ is
symmetric). Then, as $x'(k')>0$ we have
\begin{align*}
\|(\tx',\ty') - (\tx,\ty)\| &=\sum_{i \ne k,k'} |x(i)-x'(i)|
+|x(k)-1-x'(k)|+\left(x'(k')-1\right)+ \sum_{j}
|y(j)-y'(j)|\\
& \le \sum_{i \ne k'} |x(i)-x'(i)| +1+\left(x'(k')-1\right)+
\sum_{j}
|y(j)-y'(j)|\\
& = \sum_{i} |x(i)-x'(i)| + \sum_{j} |y(j)-y'(j)|=\|(x',y')-(x,y)\|.
\end{align*}
\item[(1c)] $\underline{k\ne \emptyset, k'= \emptyset,\ell=\ell'\ne\emptyset}.$
 In this case we have $x(k)>0$, and necessarily $x'(k)=0$ (otherwise $s$ would be matched with a $k$ in that system). We get
\begin{align*}
\|(\tx',\ty') - (\tx,\ty)\| &=\sum_{i \ne k} |x(i)-x'(i)|
+\left(x(k)-1\right)+ \sum_{j\ne s}
|y(j)-y'(j)|+|y(s)-(y'(s)+1)|\\
&=\sum_{i} |x(i)-x'(i)| -1 + \sum_{j\ne s}
|y(j)-y'(j)|+|y(s)-(y'(s)+1)|\\
&\le \sum_{i} |x(i)-x'(i)|  + \sum_{j} |y(j)-y'(j)|-1+1.
\end{align*}
\item[(1d)] $\underline{k=k'= \emptyset,\ell=\ell'\ne\emptyset}.$
The new $s$ is unmatched in both systems, hence
\begin{align*}
\|(\tx',\ty') - (\tx,\ty)\| &=\sum_{i} |x(i)-x'(i)| + \sum_{j\ne s}
|y(j)-y'(j)|+|(y(s)+1)-(y'(s)+1)|\\
& = \|(x',y')-(x,y)\|.
\end{align*}
\item[(1e)] $\underline{k=k'\ne \emptyset\ell\ne
\emptyset,\ell'\ne\emptyset,\ell\ne\ell'}.$ Symmetric to case (1b).
\item[(1f)] $\underline{k\ne\emptyset,k'\ne \emptyset,k\ne k',\ell\ne
\emptyset,\ell'\ne\emptyset,\ell\ne\ell'}.$ In this scenario both
$c$ and $s$ find a match in both systems, with different respective
matches in the two systems. Then, from (\ref{eq:consist}) we cannot have both
$k'\in \maQ(x,s)$ and $k \in \maQ(x',s)$. All the same, we cannot have both
$\ell'\in \maP(y,c)$ and $\ell \in \maP(y',c)$. Suppose that $x'(k)=0$ and
$y(\ell')=0$, the other cases are symmetric. We then have that
\begin{align*}
\|(\tx',\ty') - (\tx,\ty)\| & =\sum_{i \ne k,k'} |x(i)-x'(i)|
+\left(x(k)-1\right)+|x(k')-(x'(k')-1)|\\
&\quad\quad+\sum_{j \ne \ell,\ell'} |y(j)-y'(j)|
+\left(y'(\ell')-1\right)+|y(\ell)-1-y'(\ell)|\\
& =\sum_{i \ne k'} |x(i)-x'(i)|
-1+|x(k')-(x'(k')-1)|\\
&\quad\quad+\sum_{j \ne \ell} |y(j)-y'(j)|
-1+|y(\ell)-1-y'(\ell)|\\
& \le \sum_{i} |x(i)-x'(i)| + \sum_{j} |y(j)-y'(j)|-2+2.
\end{align*}
\item[(1g)] $\underline{k\ne\emptyset,k'=\emptyset,\ell\ne
\emptyset,\ell'\ne\emptyset,\ell\ne\ell'}.$ Again from (\ref{eq:consist}), we cannot have 
both $y(\ell')>0$ and $y'(\ell)>0$. Suppose that $y'(\ell)=0$ (the
other case is symmetric). Observe that we also necessarily have that
$x'(k)=0$. Thus,
\begin{align*}
\|(\tx',\ty') - (\tx,\ty)\| & =\sum_{i \ne k} |x(i)-x'(i)|
+\left(x(k)-1\right)+\sum_{j \ne \ell,\ell',s} |y(j)-y'(j)|
\\
&\quad\quad+\left(y(\ell)-1\right)+|y(\ell')-(y'(\ell')-1)|+|y(s)-(y'(s)+1)|\\
& \le\sum_{i} |x(i)-x'(i)| +\sum_{j} |y(j)-y'(j)| -2+2.
\end{align*}
\item[(1h)] $\underline{k=\emptyset,k'\ne\emptyset,\ell\ne
\emptyset,\ell'\ne\emptyset,\ell\ne\ell'}.$ Symmetric to (1g).
\item[(1i)] $\underline{k=k'=\emptyset,\ell\ne
\emptyset,\ell'\ne\emptyset,\ell\ne\ell'}.$ Say $y'(\ell)=0$ (the
case $y(\ell')=0$ is symmetric). Then we have
\begin{align*}
\|(\tx',\ty') - (\tx,\ty)\| & =\sum_{i} |x(i)-x'(i)| +\sum_{j \ne
\ell,\ell',s} |y(j)-y'(j)|\\
&\quad\quad+\left(y(\ell)-1\right)+|y(\ell')-(y'(\ell')-1)|+|(y(s)+1)-(y'(s)+1)|\\
& \le \sum_{i} |x(i)-x'(i)| +\sum_{j} |y(j)-y'(j)| -1+1.
\end{align*}
\item[(1j)] $\underline{k=k'\ne\emptyset,\ell\ne
\emptyset,\ell'=\emptyset}.$ Symmetric to (1c).
\item[(1k)] $\underline{k\ne\emptyset,k'\ne\emptyset,k\ne k',\ell\ne
\emptyset,\ell'=\emptyset}.$ Symmetric to (1g).
\item[(1$\ell$)] $\underline{k\ne\emptyset,k'=\emptyset,\ell\ne
\emptyset,\ell'=\emptyset}.$ We have $x'(k)=0$ and $y'(\ell)=0$,
hence
\begin{align*}
\|(\tx',\ty') - (\tx,\ty)\| & =\sum_{i\ne k,c} |x(i)-x'(i)| +\sum_{j
\ne
\ell,s} |y(j)-y'(j)|\\
&\quad\quad+\left(x(k)-1\right)+\left(y(\ell)-1\right)+|x(c)-(x'(c)+1)|+|y(s)-(y'(s)+1)|\\
& = \sum_{i\ne c} |x(i)-x'(i)|+\sum_{j\ne s} |y(j)-y'(j)|\\
&\quad\quad-2+|x(c)-(x'(c)+1)|+|y(s)-(y'(s)+1)|\\
& \le \sum_{i} |x(i)-x'(i)|+\sum_{j} |y(j)-y'(j)|-2+2.
\end{align*}
\item[(1m)] $\underline{k=\emptyset,k'\ne\emptyset,\ell\ne
\emptyset,\ell'=\emptyset}.$ In that case we have $x(k')=0$ and
$y'(\ell)=0$, so similarly to ($1\ell$),
\begin{align*}
\|(\tx',\ty') - (\tx,\ty)\| & =\sum_{i\ne k',c} |x(i)-x'(i)|
+\sum_{j \ne
\ell,s} |y(j)-y'(j)|\\
&\quad\quad+\left(x'(k')-1\right)+\left(y(\ell)-1\right)+|x(c)-(x'(c)+1)|+|y'(s)-(y(s)+1)|\\
& \le \sum_{i} |x(i)-x'(i)|+\sum_{j} |y(j)-y'(j)|-2+2.
\end{align*}
\item[(1n)] $\underline{k=k'=\emptyset,\ell\ne
\emptyset,\ell'=\emptyset}.$ We must have $y'(\ell)=0$, and thus
\begin{align*}
\|(\tx',\ty') - (\tx,\ty)\| & =\sum_{i\ne c} |x(i)-x'(i)| +\sum_{j
\ne
\ell,s} |y(j)-y'(j)|\\
&\quad\quad+|x(c)-(x'(c)+1)|+\left(y(\ell)-1\right)+|(y(s)+1)-(y'(s)+1)|\\
& \le \sum_{i} |x(i)-x'(i)|+\sum_{j} |y(j)-y'(j)|+1-1.
\end{align*}
\item[(1o)] $\underline{k=k'=\emptyset,\ell=\ell'=\emptyset}.$
Neither $c$ nor $s$ find a match in both systems, and thus
\begin{align*}
\|(\tx',\ty') - (\tx,\ty)\| & =\sum_{i\ne c} |x(i)-x'(i)| +\sum_{j
\ne s} |y(j)-y'(j)|\\
&\quad\quad+|(x(c)+1)-(x'(c)+1)|+|(y(s)+1)-(y'(s)+1)|\\
& \le \sum_{i} |x(i)-x'(i)|+\sum_{j} |y(j)-y'(j)|+1-1.
\end{align*}
\end{enumerate}

\medskip

$\underline{\mbox{Case 2: }s\in \maS(c).}$

\medskip

In this case $c$ and $s$
can be matched. 
We have the following sub-cases,
\begin{enumerate}
\item[(2a)] $\underline{\ell=\ell'=s}.$ In this case we must have
$k=k'=c$, in other words the entering couple is matched together in
both systems. Thus $(\tx,\ty)=(x,y)$ and $(\tx',\ty')=(x',y')$.
\item[(2b)] $\underline{\ell=s,\ell'\ne s,\,k'\ne \emptyset}.$
Thus we have $k=c$, and according to the buffer-first rule we must
then have $x(k')=0$ and $y(\ell')=0$. Thus,
\begin{align*}
\|(\tx',\ty') - (\tx,\ty)\| & =\sum_{i\ne k'} |x(i)-x'(i)| +\sum_{j
\ne \ell'} |y(j)-y'(j)|+(x'(k')-1)+(y'(\ell')-1)\\
& = \sum_{i} |x(i)-x'(i)|+\sum_{j} |y(j)-y'(j)|-2.
\end{align*}
\item[(2c)] $\underline{\ell=s,\ell'\ne s,k'=\emptyset}.$ Again, we have
$y(\ell')=0$ and thus
\begin{align*}
\|(\tx',\ty') - (\tx,\ty)\| & =\sum_{i} |x(i)-x'(i)| +\sum_{j
\ne \ell',s} |y(j)-y'(j)|+|y(s)-(y'(s)+1)|+(y'(\ell')-1)\\
& \le \sum_{i} |x(i)-x'(i)|+\sum_{j} |y(j)-y'(j)|+1-1.
\end{align*}
\item[(2d)] $\underline{\ell\ne s,\ell'= s,k \ne \emptyset}.$ Symmetric to (2b).
\item[(2e)] $\underline{\ell\ne s,\ell'= s,k =\emptyset}.$ Symmetric to (2c).
\item[(2f)] $\underline{k=c,k'\ne c,\,\ell'\ne \emptyset}.$ Analog
to (2b).
\item[(2g)] $\underline{k=c,k'\ne c,\,\ell'=\emptyset}.$ Analog
to (2c).
\item[(2h)] $\underline{k\ne c,k'= c,\,\ell\ne \emptyset}.$ Analog
to (2d).
\item[(2i)] $\underline{k\ne c,k'= c,\,\ell=\emptyset}.$ Analog
to (2e).
\item[(2j)] $\underline{\ell\ne s,\ell'\ne s}.$ We are back to Case 1, excluding the cases where $k=\ell=\emptyset$ and the cases where
$k'=\ell'=\emptyset$.
\end{enumerate}

\section{Proof of Proposition \ref{prop:nonexp2}}

Concerning the policy {\sc ml}, the proof is similar to that for random policies, 
except for the consistency property (\ref{eq:consist}), which does not hold in this case. 
Specifically, an entering item can be matched
with items of two different classes in the two systems, whereas the
queues of these two classes are non-empty in both systems.

Let us consider the following case (the other ones are symmetric or
similar)~: $c \not\in \maC(s)$, the customer of class $c$ is matched
with a server of class $\ell$ in both systems, whereas $s$ is
matched with a customer of class $k$ in the first system, with a
customer of class $k'\ne k$ in the second, with 
\[\{k,k'\} \subset \maQ(x,s) \cap \maQ(x',s).\]
Thus
\begin{equation}
\label{eq:losers1} \|(x',y')\ccc_{\textsc{ml}}(c,s) -
(x,y)\ccc_{\textsc{ml}}(c,s)\| =\sum_{i \ne k,k'}
|x(i)-x'(i)|+ \sum_{j\ne s} |y(j)-y'(j)|+R, \end{equation} where
\[R=\left|(x(k)-1)-x'(k)\right|+\left|x(k')-\left(x'(k')-1\right)\right|.\]
We are in the following alternatives,
\begin{enumerate}
\item if $x(k) > x'(k)$ and $x'(k') > x(k')$, then
\begin{equation*}
R=\left(x(k)-1-x'(k)\right)+\left(x'(k')-1-x(k')\right)=\left|x(k)-x'(k)\right|+\left|x(k')-x'(k')\right|-2.\end{equation*}
\item if $x(k) \le x'(k)$ and $x'(k') > x(k')$, then
\begin{equation*}
R=\left(x'(k)-x(k)+1\right)+\left(x(k')-1-x'(k')\right)
 =\left|x(k)-x'(k)\right|+\left|x(k')-x'(k')\right|.\end{equation*}
\item if $x(k) > x'(k)$ and $x'(k') \le x(k')$, we also have
\begin{equation*}
R=\left(x(k)-1-x'(k)\right)+\left(x(k')-x'(k')+1\right)=\left|x(k)-x'(k)\right|+\left|x(k')-x'(k')\right|.
 \end{equation*}
 \item the case $x(k) \le x'(k)$ and $x'(k') \le x(k')$ cannot
 occur. Indeed, by the definition of {\sc ml} we have that
\[x(k') \le x(k)\,\mbox{ and }x'(k) \le x'(k').\]
Thus, would we also have $x(k) \le x'(k)$ and $x'(k') \le x(k')$, we
would obtain that
\[x(k) \le x'(k) \le x'(k') \le x(k') \le x(k),\]
and thus
\[x(k)=x(k')= x'(k) = x'(k').\]
This is impossible since, in that case, both systems would have
chosen the same match $k$ (if $\phi$ prioritizes $k$ over $k'$) or
the same match $k'$ (else), for the new $s$-server.
\end{enumerate}
Consequently, in view of (\ref{eq:losers1}), in all possible cases
we obtain that \begin{multline*}
\|(x',y')\ccc_{\textsc{ml}}(c,s) -
(x,y)\ccc_{\textsc{ml}}(c,s)\|\\
\le \sum_{i \ne k,k'} |x(i)-x'(i)|+ \sum_{j\ne s}
|y(j)-y'(j)|+\left|x(k)-x'(k)\right|+\left|x(k')-x'(k')\right|=\|(x',y')-
(x,y)\|,
\end{multline*} 
which concludes the proof.

\bibliographystyle{abbrv}
\bibliography{matching}

\end{document}